\documentclass[twoside]{amsart}

\usepackage{amsmath,amssymb,amsfonts,amsthm,latexsym}
\usepackage{amscd,graphicx,color,enumerate}
\usepackage{hyperref}
\usepackage{mathrsfs}
\usepackage{bbm}
\numberwithin{equation}{section}

\newtheorem{theorem}{Theorem}[section]
\newtheorem{lemma}[theorem]{Lemma}
\newtheorem{proposition}[theorem]{Proposition}
\newtheorem{corollary}[theorem]{Corollary}

\newtheorem{mainthm}{Theorem}
\newtheorem{mainprop}[mainthm]{Proposition}

\theoremstyle{definition}
\newtheorem{definition}[theorem]{Definition}
\newtheorem{example}[theorem]{Example}
\theoremstyle{remark}
\newtheorem{remark}[theorem]{Remark}
\newtheorem{remarks}[theorem]{Remarks}

\newcommand{\E}{\mathcal {E}}
\newcommand{\G}{\mathcal {G}}
\newcommand{\C}{\mathbb{C}}

\newcommand{\Q}{\mathbb{Q}}
\newcommand{\R}{\mathbb{R}}

\newcommand{\F}{\mathbb{F}}
\newcommand{\Z}{\mathbb{Z}}

\newcommand{\SL}{\operatorname{SL}}
\newcommand{\SP}{\operatorname{Sp}}

\newcommand{\SO}{\operatorname{SO}}
\newcommand{\OO}{\mathcal {O}}
\newcommand{\GL}{\operatorname{GL}}

\newcommand{\Span}{\operatorname{Span}}

\newcommand{\Ker}{\operatorname{Ker}}

\newcommand{\Rank}{\operatorname{rank}}

\newcommand{\co}{\colon\thinspace}
\newcommand{\git}{/\!\!/}

\newcommand{\calO}{\mathcal{O}}
\newcommand{\pr}{\mathrm{pr}}
\newcommand{\sm}{\mathrm{sm}}
\newcommand{\sing}{\mathrm{sing}}
\newcommand{\sg}{\mathrm{sg}}
\newcommand{\codim}{\operatorname{codim}}

\newcommand{\Spin}{\operatorname{Spin}}
\newcommand{\Hom}{\operatorname{Hom}}
\newcommand{\inv}{^{-1}}

\newcommand{\NN}{\mathcal N}
\newcommand{\Ad}{\operatorname{Ad}}
\newcommand{\Ann}{\operatorname{Ann}}

\newcommand{\Lie}{\operatorname{Lie}}
\newcommand{\lie}[1]{{\mathfrak #1}}
\newcommand{\lieg}{\lie g}
\newcommand{\lieh}{\lie h}
\newcommand{\liep}{\lie p}
\newcommand{\X}{\mathscr X\!}
\newcommand{\Pic}{\operatorname{Pic}}

\newcommand{\TC}{\operatorname{TC}}
\newcommand{\red}{{\operatorname{red}}}
\newcommand{\B}{\mathsf B}

\renewcommand{\Re}{\operatorname{Re}}
\renewcommand{\k}{\mathbbm{k}}
\begin{document}

\title[Zero fiber of moment map \& rational singularities]{When does the zero fiber of the moment map have rational singularities?}

\author[H.-C. Herbig]{Hans-Christian Herbig}
\address{Departamento de Matem\'{a}tica Aplicada, Universidade Federal do Rio de Janeiro,
Av. Athos da Silveira Ramos 149, Centro de Tecnologia - Bloco C, CEP: 21941-909 - Rio de Janeiro, Brazil}
\email{herbighc@gmail.com}

\author[G. W. Schwarz]{Gerald W. Schwarz}
\address{Department of Mathematics, Brandeis University,
Waltham, MA 02454-9110, USA}
\email{schwarz@brandeis.edu}

\author[C. Seaton]{Christopher Seaton}
\address{Department of Mathematics and Computer Science,
Rhodes College, 2000 N. Parkway, Memphis, TN 38112, USA}
\email{seatonc@rhodes.edu}

\subjclass[2020]{Primary 53D20, 14B05; Secondary 13A50, 13H10, 14M35, 20G20}
\keywords{singular symplectic reduction, moment map, rational singularities, symplectic singularity, representation variety, character variety, representation growth of linear groups}

\thanks{All authors were supported by a Research in Pairs grant from
CIRM (Luminy) and a Bernoulli Brainstorm grant from the
Centre Interdisciplinaire Bernoulli
of the EPFL.
H.-C.H. was supported by CNPq through the \emph{Plataforma Integrada Carlos Chagas}.
C.S. was supported by the E.C.~Ellett Professorship in Mathematics.}

\begin{abstract}
Let $G$ be a complex reductive group and $V$ a $G$-module. There is a natural moment mapping $\mu\colon V\oplus V^*\to\lieg^*$ and we denote $\mu\inv(0)$  (the shell) by
$N_V$.
We use invariant theory and results of Musta\c{t}\u{a} \cite{MustataJetRational} to find 
criteria for $N_V$ to have rational singularities and for the categorical quotient
$N_V\git G$ to have symplectic singularities, the latter results improving upon \cite{HerbigSchwarzSeaton2}. It turns out that  for ``most'' $G$-modules $V$, the shell $N_V$ has rational singularities. For the case of direct sums of classical representations of the classical groups, $N_V$ has rational singularities and $N_V\git G$ has symplectic singularities if $N_V$ is a reduced and irreducible complete intersection.
Another important special case is $V=p\,\lieg$ (the direct sum of $p$ copies of the Lie algebra of $G$) where $p\geq 2$. We  show that $N_V$ has rational singularities and that $N_V\git G$ has symplectic singularities, improving  upon results of \cite{BudurRational},
\cite{AizenbudAvniRepGrowth},
\cite{kapon2019singularity} and \cite{GlazerHendel}.
Let $\pi=\pi_1(\Sigma)$ where $\Sigma$ is a closed Riemann surface of genus $p\geq 2$. Let $G$ be semisimple and let $\Hom(\pi,G)$ and $\X(\pi,G)$ be  the corresponding representation variety and character variety.
We show that $\Hom(\pi,G)$  is a complete intersection with rational singularities and  that $\X(\pi,G)$ has symplectic singularities. If $p>2$ or  $G$ contains no simple factor of rank $1$, then the singularities of $\Hom(\pi,G)$ and $\X(\pi,G)$ are in codimension at least four and $\Hom(\pi,G)$ is locally factorial. If, in addition, $G$ is simply connected, then $\X(\pi,G)$ is locally factorial.
\end{abstract}

\maketitle

\tableofcontents


\section{Introduction}
\label{sec:Intro}

Let $G$ be a reductive complex group and $V$ a $G$-module. Then there is a natural $G$-invariant symplectic form $\omega$ on $V\oplus V^*$ such that $V$ and $V^*$ are isotropic subspaces. Let $\lieg$ denote the Lie algebra of $G$. Then there is a canonical moment mapping  $\mu\colon V\oplus V^*\simeq T^*V\to\lieg^*$. If $(v,v^*)\in V\oplus V^*$ and $A\in\lieg$, then
$\mu(v,v^*)(A)=v^*(A(v))$.
Let $N_V$ denote $\mu\inv(0)$ which we call \emph{the shell.} In  \cite{HerbigSchwarzSeaton2}  we found criteria for $N_V\git G$ to have symplectic singularities. In some examples, e.g., if $G$ is a torus, proving that $N_V\git G$ had symplectic singularities involved first proving that $N_V$ had rational singularities. This was also established in \cite{CapeHerbigSeaton} for $G=\SO_n(\C)$ and $V=k\,\C^n$, $k\geq n$. The main theorem of Budur \cite{BudurRational} proves that $N_V$ has rational singularities when $V=p\,\lieg$, $p\geq 2$, and $G=\GL_n(\C)$. Aizenbud and Avni \cite{AizenbudAvniRepGrowth} prove the same for $G$ semisimple, but require larger $p$.
The bounds were  improved  in \cite{kapon2019singularity} and to $p\geq 4$ in \cite{GlazerHendel}.
This all sparked our interest in determining when a general $N_V$ has rational singularities.

Budur uses work  of Musta\c{t}\u{a}   \cite{MustataJetRational} which gives a criterion for a local complete intersection
variety
to have rational singularities in terms of its jet schemes.
We use this criterion throughout our paper. A first consequence is the following (Corollary \ref{cor:dimGMod}, Remark \ref{rem:dimGModGeneric}).

\begin{mainthm}\label{thm:main5}
Let $G$ be semisimple and consider $G$-modules $V$ such that $V^G=0$ and each irreducible factor of $V$ is an almost faithful $G$-module. Then  there are only finitely many isomorphism classes of such $G$-modules such that $N_V$ does not have rational singularities.
\end{mainthm}

The theorem above, however, is not very useful when one is presented with a specific $G$ and $V$.

We say that $N_V$ has
\emph{FPIG} (finite principal isotropy groups) if it has an open dense subset of closed orbits with finite isotropy group.
When $N_V$ is a complete intersection with FPIG and rational singularities, we say that $N_V$ is
\emph{CIFR}.
We use   \cite{MustataJetRational} and the symplectic slice theorem of \cite{HerbigSchwarzSeaton2} to find criteria for $N_V$ to
be
CIFR in terms of the symplectic slice representations of $N_V$.
For 
more about 
the following see the discussion  preceding Theorem \ref{thm:symplectic.slice}.
 Let $x\in N_V$ such that $Gx$ is closed. Then the isotropy group $H=G_x$ is reductive and we have a symplectic slice representation $H\to\GL(S)$ where $S\subset V\oplus V^*$ is an $H$-submodule and the restriction $\omega_S$ of $\omega$ to $S$ is non-degenerate (and $H$-invariant).   We can decompose $S$ as $S^H\oplus S_0$ where $\omega_S$ is non-degenerate on $S_0$ and admits Lagrangian $H$-submodules,
i.e., isotropic submodules of dimension $(1/2)\dim S_0$. Let $W_0$ be such a submodule and let
$N_0 := N_{W_0}$ denote the shell of $W_0\oplus W_0^*$. The shell $N_0$ only depends upon the restriction of $\omega_S$ to $S_0$
and not the choice of $W_0$ (Lemma \ref{lem:mu.formula}).
Let $\NN(S_0)$ denote the null cone of $S_0$, 
the union of the $H$-orbits with $0$ in their closure, 
 and let $\NN(N_0) = N_0\cap\NN(S_0)$.
Let $(N_0)_m$ denote the $m$th jet scheme of $N_0$ and let $\rho_m\colon (N_0)_m\to N_0$ denote the natural projection; see Section \ref{subsec:BackReps}.
Let $(N_0)_m'$ denote the closure of $\rho_m\inv(N_0\setminus (\NN(N_0)\cap (N_0)_\sing))$
where $(N_0)_\sing$ denotes the singular points of $N_0$. 
We say that $\NN(N_0)$ is \emph{irrelevant\/} if
$(N_0)_m'=(N_0)_m$ for all $m\geq 1$.

We then have the following criterion for the shell to have rational singularities (Theorem \ref{thm:iff}).

\begin{mainthm}\label{thm:mainIFF}
Let $V$ be a $G$-module and $N = N_V$. Assume that $N$ is a normal complete intersection with FPIG. Then $N$ has rational singularities,  hence is CIFR,
if and only if $\NN(N_0)$ is irrelevant for every symplectic slice representation $(S,H)$ of $N$.
\end{mainthm}

For many cases of $G$ and $V$, the dimensions of the $\rho_m\inv(\NN(N_0))$ are already small enough to establish the irrelevance of the
$\NN(N_0)$, and we frequently use this approach. The main novel technique of this paper is to use bounds on the dimension of linear
subspaces of the null cone  $\NN(W_0)$ 
 to show that $\rho_m\inv(\NN(N_0))$ is nowhere dense in the jet scheme $(N_0)_m$; this implies
the irrelevance of $\NN(N_0)$
 and hence that $N$ has rational singularities by Theorem \ref{thm:mainIFF}. 
More specifically, let $m_0(W_0)$ denote the
maximal dimension of a  linear subspace of $\NN(W_0)$. One of  our main  results is the following (Theorems  \ref{N.good.tori} and \ref{thm:use.E_m}).

\begin{mainthm}\label{thm:main1}
Let $V$ be a $G$-module and  suppose that for all symplectic slice representations $(S=S^H\oplus S_0,H)$   of $N_V$ where $\dim H>0$ there is a choice of Lagrangian $H$-submodule $W_0$ of $S_0$ such that either
\begin{enumerate}
\item  the group $H^0$ is a torus and $W_0$
has FPIG, or
\item   $m_0(W_0)<\dim W_0-\dim H$.
\end{enumerate}
Then each $\NN(N_0)$ is irrelevant and $N_V$ is CIFR.
\end{mainthm}

In general, the calculation of   $m_0(W_0)$ is difficult. However, if $W_0$ is an orthogonal $H$-module,  we may replace (2) by
\begin{enumerate}
\item[\emph{(2')}]  $\dim H<(1/2)(\dim W_0-\dim W_0^T)$ where $T$ is a maximal torus of $H$.
\end{enumerate}
This follows from the following general fact (Proposition \ref{prop:dim.L}).

\begin{mainprop}\label{prop:D}
Let $V$ be an orthogonal $G$-module. Then
$$
    m_0(V)= (1/2)(\dim V-\dim V^T)
$$
where $T$ is a maximal torus of $G$.
\end{mainprop}

If $V=\lieg$ where $G$ is simple, then the proposition says that $m_0(\lieg)$ is the dimension of a maximal nilpotent subalgebra of $\lieg$. This result is due to Gerstenhaber \cite{Gerstenhaber} for $\lie sl_n$ and Meshulam-Radwan \cite{Meshulam-Radwan} in general.
See also \cite{Draisma-Kraft-Kuttler} by Draisma, Kraft and Kuttler. This paper provided a key idea in our proof of Proposition \ref{prop:D}.
If $V$ is an orthogonal $G$-module, then in Theorem \ref{thm:main1} we may always choose $W_0$ to be an orthogonal $H$-module (Lemma \ref{lem:slice.orthog.rep} and Proposition \ref{prop:orthog.reps}).

Condition (2) of Theorem~\ref{thm:main1} is stronger than the equivalence given in Theorem~\ref{thm:mainIFF}; for instance, the $\SL_2$-module $V=3\C^2$
does not satisfy (2) yet
$N_V$ is CIFR. However, Theorem~\ref{thm:main1} along with condition (2') in the orthogonal case
(and some ad hoc arguments)
is sufficient to establish
that the classical representations of the classical groups, as well as several other important cases of $G$ and $V$, have shells that are CIFR whenever the
shells are complete
intersection varieties,
 i.e., whenever Musta\c{t}\u{a}'s criterion for rational singularities
can even be applied;
see Section~\ref{sec:Examples}.
In addition, it allows us to handle the case of two or more copies of the adjoint representation, yielding the following (Theorem  \ref{thm:p.lieg}).
We say that a complex affine variety is \emph{factorial\/} if $\C[X]$ is a UFD.

\begin{mainthm}\label{thm:main2}
Assume that $G$ is semisimple and let $V=p\,\lieg$ where $p\geq 2$. Then
\begin{enumerate}
\item $N_V$ is CIFR.
\item  If $p>2$ or $G$ contains no simple factor of rank $1$, then $N_V$ and $N_V\git G$ are factorial.
\end{enumerate}
\end{mainthm}

Theorem \ref{thm:symplectic.sings} implies the following.

\begin{mainthm}\label{thm:main4}
Let $V$ be as in Theorem \ref{thm:main1} or \ref{thm:main2}. Then   $N_V\git G$ is graded Gorenstein with symplectic singularities.
\end{mainthm}

Representations of quivers have associated moment mappings as well, and our criteria can be applied to show that the corresponding zero fibers have rational singularities
in many cases.

Let $\pi=\pi_1(\Sigma)$ where $\Sigma$ is a compact Riemann surface of genus $p>1$. Let $G$ be semisimple, let $\Hom(\pi,G)$ be the corresponding representation variety, and let $\X(\pi,G)$ be the corresponding character variety. Let $\rho_0\in\Hom(\pi,G)$ denote the trivial homomorphism.
We use Theorem \ref{thm:main2} to show that the tangent cone to  $\Hom(\pi,G)$ at a closed $G$-orbit  has rational singularities and we establish the following  (Theorem \ref{thm:char.vars}).

\begin{mainthm}\label{thm:main3}
Let $G$ and $\pi$ be as above.
\begin{enumerate}
\item  $\Hom(\pi,G)$ is
CIFR and each irreducible component  has dimension $(2p-1)\dim G$.
\item $\X(\pi,G)$ has symplectic singularities and  each irreducible component has dimension $(2p-2)\dim G$.
\end{enumerate}
Now suppose that $p>2$ or that every simple component of $G$ has rank at least $2$.
\begin{enumerate}
\addtocounter{enumi}{2}
\item The singularities of $\Hom(\pi,G)$ are in codimension
at least four and $\Hom(\pi,G)$ is locally factorial.
\item The singularities of $\X(\pi,G)$ are in codimension
at least  four and the irreducible component containing $G\rho_0$ is locally factorial. In particular, if $G$ is simply connected, then $\X(\pi,G)$ is locally factorial.
\end{enumerate}
\end{mainthm}

As in Budur  \cite[Theorem 1.10]{BudurRational}, using results of Simpson \cite[p.\ 69]{SimpsonModuliII}, Theorem \ref{thm:main3}
shows that the Betti, de Rham and Dolbeault representation spaces of principal $G$-bundles on our Riemann surface $\Sigma$ have rational singularities and that the corresponding moduli spaces have symplectic singularities.
Here $G$ is semisimple. We also prove a version of Theorem \ref{thm:main3} that handles the case of $G$ reductive (Theorem \ref{thm:char.vars.red.gp}).
Compare
\cite[Theorem 1.1]{LawtonManon} for similar results in the case when $\pi$ is a free group as well as the related work \cite{kapon2019singularity}
on graph varieties
and  \cite{GlazerHendel} on word maps.

For a topological group $\Gamma$, let
$r_n(\Gamma)$ denote the number of $n$-dimensional irreducible continuous complex representations of $\Gamma$ (the \emph{growth sequence of $\Gamma$}) and
let $R_n(\Gamma) = \sum_{i=1}^n r_i(\Gamma)$.
The \emph{representation zeta function} $\zeta_\Gamma(s)$ of $\Gamma$ is given by
\[
    \zeta_\Gamma(s) =   \sum\limits_{i=1}^\infty r_n(\Gamma)n^{-s}.
\]
The \emph{abscissa of convergence of $\zeta_\Gamma(s)$\/} is
\[
    \alpha(\Gamma)  :=   \limsup\limits_{n\to\infty} \frac{\log R_n(\Gamma)}{\log n}.
\]
Then $\zeta_\Gamma(s)$ converges absolutely for $\{\Re(s)>\alpha(\Gamma)\}$.

We now apply some remarkable results of Aizenbud and Avni relating growth sequences of groups and rational singularities of varieties $\Hom(\pi,G)$.
From Theorem \ref{thm:main3}(1) and \cite[Theorem IV]{AizenbudAvniRepGrowth} we have the following.
\begin{mainthm}
\label{thm:abscissa1}
Let $\k$ be a
finitely generated field of characteristic zero. Let   $G$ be a semisimple group defined over $\k$ and $\F$   a   local field containing $\k$. Let $\Gamma$ be
a compact open subgroup of $G(\F)$. Then $\alpha(\Gamma)<2$.
\end{mainthm}

The above result is in \cite{AizenbudAvniRepGrowth,AizenbudAvniCounting}  with a larger bound on  $\alpha(\Gamma)$
that was improved in
\cite[Corollary 1.9]{GlazerHendel}.
For  $G = \SL_n$, the estimate above was established by Budur   \cite[Theorem 1.7]{BudurRational}).
Similarly, Aizenbud and Avni give bounds for abscissae related to arithmetic subgroups of high rank semisimple groups.  Theorem \ref{thm:main3} allows one to improve these bounds. For example, one has the following version of \cite[Theorem B]{AizenbudAvniCounting} (established for $\SL_n$ in
\cite[Theorem 1.4]{BudurRational}).
 \begin{mainthm}
\label{thm:abscissa2}
Let $G$ be an affine group scheme over $\Z$ whose generic fiber $G_\Q$ is (almost) $\Q$-simple, connected, simply connected and of $\Q$-rank at least $2$. Then $\alpha(G(\Z))\leq 2$.
\end{mainthm}

 \begin{remark}
 Suppose that $\Gamma$ is one of the groups  considered above and $\alpha(\Gamma)\leq 2$.   Then using a generalization of Faulhaber's formula
\cite{McGownParks} one gets estimates on the growth of $r_n(\Gamma)$; namely, $r_n(\Gamma) = O(n^{1+\epsilon})$ for every $\epsilon > 0$. If $\alpha(\Gamma)<2$, one gets that  $r_n(\Gamma) = O(n^{1-\epsilon})$ for some $\epsilon > 0$.
 \end{remark}

The outline of this paper is as follows.
In \S \ref{sec:Back} we provide background material as well as criteria for $N_V$ to be a normal complete intersection with FPIG.
We introduce  Musta\c{t}\u{a}'s criterion for rational singularities and establish Theorem \ref{thm:mainIFF}.
In \S \ref{sec:JetShell}
we study the jet schemes of $N_V$ and establish
Theorems  \ref{thm:main5}  and \ref{thm:main1}.
In \S \ref{sec:Adjoint} we apply our results to the case of copies of the adjoint representation and establish Theorem \ref{thm:main2}. In \S \ref{sec:symplectic.sings} we
 establish Theorem \ref{thm:main4}. In \S \ref{sec:Examples} we apply Theorem \ref{thm:main1} to classical representations of the classical groups. In \S \ref{sec:Applications} we consider applications to representation varieties and character varieties and establish
 Theorem  \ref{thm:main3}.
 An appendix
 \S \ref{sec:Appendix} is devoted to showing that certain maps arising in \S \ref{sec:Applications} are moment mappings.

\medskip
We use some standard and not so standard abbreviations in this paper. Particularly important are three already mentioned in the introduction, which we summarize here. The notion of a $G$-variety having FPIG (finite principal isotropy groups) is defined in \S \ref{sec:G-modules}, the notion of a $G$-module having a property UTCLS (up to choosing a Lagrangian submodule) is given in Definition \ref{def:UTCLS}, and a $G$-variety is CIFR if it is a complete intersection with FPIG  and rational singularities.


\section*{Acknowledgements}
We thank CIRM Luminy (RIP) and  the CIB of the EPFL (Bernoulli Brainstorm) for their support which enabled us to begin  this research, together,  in January of 2020.  We thank Simon Lyakhovich for helpful comments and Nero Budur for prompt and helpful responses to questions.
We also thank Itay Glazer, Sean Lawton, and Travis Schedler for pointing out important connections with the literature
as well as referees for helpful comments on an earlier version of this manuscript.


\section{Background}
\label{sec:Back}


\subsection{$G$-modules and the shell}\label{sec:G-modules}

Let $G$ be a reductive complex group and  $X$ an affine $G$-variety. (We assume that varieties are reduced but not necessarily irreducible except in Section~\ref{subsec:BackReps}.)\ The categorical quotient $Z=X\git G$ is the affine variety with coordinate ring $\C[X]^G$ and $\pi\colon X\to Z$, the quotient mapping, is dual to the inclusion $\C[X]^G\subset\C[X]$. The categorical quotient parameterizes the closed orbits in $X$.
A subset of $X$ is said to be \emph{$G$-saturated\/} if it is a union of fibers of $\pi$.
 If $x\in X$ and $Gx$ is closed, then the isotropy group $G_x$ is reductive. For a reductive subgroup $H$ of $G$, we let $Z_{(H)}$ denote the set of closed orbits in $X$ whose isotropy groups are in the conjugacy class $(H)$ of $H$. We write $(H)< (H')$ if $H$ is conjugate to a proper subgroup of $H'$. The $Z_{(H)}$ are called \emph{isotropy strata\/} of $Z$. They are locally closed in the Zariski topology. Assume that $Z$ is irreducible. Then there is an open dense stratum $Z_\pr$, the \emph{principal stratum\/}. Corresponding closed orbits are called \emph{principal orbits\/} and any corresponding reductive subgroup $H$ of $G$ is called a  \emph{principal isotropy group}. We let $X_\pr$ denote $\pi\inv(Z_\pr)$.  If $Z_{(H')}$ is not empty, then $(H)\leq (H')$, i.e., a principal isotropy group $H$ is conjugate to a subgroup of $H'$.  The $G$-action on $X$ is \emph{stable\/} if $\pi\inv(Z_\pr)$ consists of closed orbits. We say that $X$ has FPIG  if the principal isotropy groups are finite.

 Let $V$ be a $G$-module and $x\in V$ such that $Gx$ is closed. Let $H=G_x$. Then, as 
 an 
 $H$-module, $V=T_x(Gx)\oplus W$ where $W$ is an $H$-module, which is called the \emph{slice representation of $H$} \cite{LunaSlice}. Since $T_x(Gx)\simeq \lieg/\lieh$ where
 $\lieh$ denotes the Lie algebra of $H$, $W$ is completely determined by $H$.

 Let $V$ be a $G$-module  and set $U=V\oplus V^*$. Then the \emph{standard symplectic form $\omega$ on $U$\/} and \emph{standard moment mapping\/} $\mu\colon U\to \lieg^*$ are given by
 \begin{align*}
\omega((v,v^\ast),(w,w^\ast))&=w^\ast(v)-v^\ast(w),\quad (v,v^\ast),\ (w,w^\ast)\in  V\oplus V^\ast.\\
 \mu(v,v^\ast)(A)  &=v^\ast\big(A(v)\big),\quad (v,v^*)\in V\oplus V^\ast,\   A\in\lieg.
\end{align*}
Of course, moment mappings are only unique up to a constant in $\Ann[\lieg,\lieg]\subset\lieg^*$, so  our  $\mu$ is standard in the sense that  $\mu(0)=0$.

\begin{lemma}\label{lem:mu.formula}
Let $(U,\omega)$ be a symplectic $G$-module
with a moment mapping vanishing at $0\in U$.
 If $U$   admits a Lagrangian $G$-submodule $V$, then there is an isomorphism $U\simeq V\oplus V^*$  under which the symplectic form    and moment mapping  become  standard.
\end{lemma}
\begin{proof}
Since $V$ is Lagrangian, $\omega$ induces a non-degenerate pairing  $V\times (U/V)\to\C$  so that $U/V\simeq V^*$. Since $G$ is reductive, $U\simeq V\oplus V^*$ where $V^*$ is Lagrangian and the usual pairing of $V$ and $V^\ast$ is induced by $\omega$. With $U$ viewed as $V\oplus V^\ast$, $\omega$ and $\mu$ are standard.
\end{proof}

The \emph{shell $N_V$ of $V$}, which we denote simply $N$ if $V$ is clear from the context, is the fiber $\mu\inv(0)\subset U$. The shell may be neither reduced nor irreducible. In \cite{HerbigSchwarzSeaton2} we called $N_V$ the \emph{complex\/} shell since we were also dealing with real moment mappings. In this paper we only deal with the complex case.
A big part of our task is  to deduce properties of $N_V$ from those of $V$.

We now recall (a weak form of) the symplectic slice theorem of
\cite[\S 3.4]{HerbigSchwarzSeaton2}.
Let $x\in N_V$ such that $Gx$ is closed. 
The isotropy group
$H=G_x$ is reductive and we let $E$ denote $T_x(Gx)$.  Then $E$ is isotropic  and $\omega$ is non-degenerate on $E^\perp/E$ where  $\perp$ denotes perpendicular with respect to $\omega$. Since $H$ is reductive,
 there is an $H$-module $S$ such that  $E^\perp= S\oplus E$ and $U\simeq E\oplus E^*\oplus S$  where $\omega$ induces a symplectic form $\omega_S$ on $S$.

 Note that $H$ determines $S$ since $E\simeq \lieg/\lieh\simeq E^*$. By \cite[Lemma 3.10]{HerbigSchwarzSeaton2}, $S$ admits a Lagrangian $H$-submodule $W$ so that $\omega_S$ is standard on $S\simeq W\oplus W^*$.
 Let $N_W$ denote the corresponding shell in $S$.

Recall that if $\phi\colon X\to Y$ is a $G$-equivariant mapping of $G$-varieties, then for any $x\in X$, $G_x\subset G_{\phi(x)}$. The mapping $\phi$ is \emph{isovariant\/} if it is equivariant and $G_x=G_{\phi(x)}$ for all $x\in X$.

\begin{theorem}[Symplectic Slice Theorem]\label{thm:symplectic.slice}
There is an $H$-saturated affine neighborhood $Q$ of $0\in E^*\oplus S$ and an \'etale $G$-isovariant mapping $\phi\colon G\times^HQ\to V\oplus V^*$ where $[e,0]$ is sent to $x$.  The image of $\phi$ is $G$-saturated and open. Pulling back the standard symplectic form and moment mapping on $V\oplus V^*$ by $\rho$, the shell of $G\times^HQ$ is $G\times^H(Q\cap (\{0\}\times N_W)$. The induced mapping
$$
\rho\inv(N_V)\simeq G\times^H(Q\cap (\{0\}\times N_W))\to N_V
$$
is  an \'etale  mapping of affine schemes and is the restriction of a $G$-isovariant map.
\end{theorem}

\begin{remark}
The theorem in \cite{HerbigSchwarzSeaton2} is stated in the case that $N_V$ is an irreducible variety but the proof does not need this and implies the version above.
\end{remark}

\begin{corollary}\label{cor:dim.S}
Let $Gx\subset N_V$ be a closed orbit with symplectic slice representation $(S,H)$.  Then
$$
\dim_x N_V=\dim G-\dim H+\dim_0 N_W.
$$
\end{corollary}

\begin{corollary}\label{cor:property(P)}
\label{cor:P}
Let (P) be one of the following conditions: reduced, smooth, normal
or rational singularities.
Then $N_V$ satisfies (P) at $x$ if and only if $N_W$ satisfies (P) at $0$.
\end{corollary}
\begin{remark}\label{rem:closed.orbits}
Let (P) be one of conditions normal or rational singularities. Suppose that an affine $G$-variety $X$ satisfies (P) along every closed orbit. Let  $\Omega=\{x\in X\mid\text{ (P) holds at }x\}$. Then $\Omega$ is open and $G$-stable. If $\Omega\neq X$, then $\Omega^c$ is $G$-stable and closed, hence contains a closed $G$-orbit, which is a contradiction. Hence $X$ satisfies (P).
\end{remark}
\subsection{Conditions on affine $G$-varieties.}

 Let $X$ be an affine   $G$-variety whose  quotient $Z$ is irreducible.  For $k\geq 0$, we say that $X$ is
\emph{$k$-principal} if $\codim  X\smallsetminus X_{\pr} \geq k$.
Let  $X_{(r)}$ denote the locally closed set of points with isotropy group of   dimension $r$. We say that
$X$ is \emph{$k$-modular} if $\codim_X  X_{(r)}\geq r+k$ for $1\leq r\leq  \dim  G$.
Note, in particular,  if $X$ is $k$-modular, then $X_{(0)}\neq\emptyset$.
We say that $X$ is \emph{$k$-large} if it is $k$-principal, $k$-modular, and has FPIG.
See \cite{GWSlifting} for more background on these concepts, and note that references
sometimes differ about whether FPIG is required as part of the definition of
$k$-modular or $k$-principal.
We say that a $G$-module $V$ is \emph{orthogonal} if $V$ admits a non-degenerate symmetric $G$-invariant bilinear form.
Note that if $V$ is orthogonal, then it is stable \cite{LunaSlice}.

\begin{definition}
\label{def:UTCLS}
Let $V$ be a $G$-module and $U=V\oplus V^*$ with the standard symplectic form.
We say that $V$ has property $(P)$ \emph{UTCLS} (\emph{up to choosing a Lagrangian submodule})
if   there is a  Lagrangian $G$-submodule $V'$  of $U$ with property $(P)$.
\end{definition}
By Lemma \ref{lem:mu.formula}, $U\simeq V'\oplus (V')^*$ where the shell remains the same.

We recall the following; see \cite[Proposition~9.4]{GWSlifting}, \cite[Proposition~3.2]{HerbigSchwarzSeaton2},
and for (6), \cite[Theorem~2.4]{PanyushevJacob} and \cite[Proposition~6]{Avramov}.

\begin{proposition}\label{prop:props.of.N}
Let $U=V\oplus V^*$ with the standard symplectic form and shell
$N = N_V$. Let $R = \{x\in U \mid G_{x}$ is finite$\}$.
\begin{enumerate}
\item   The shell $N$ is a complete intersection (hence Cohen-Macaulay) if and only if $V$ is $0$-modular if and only if   $\dim N=2\dim V-\dim G$.
\end{enumerate}
Now assume that $V$ is $0$-modular.
\begin{enumerate}
\addtocounter{enumi}{1}
\item   The set $R$ equals $\{x\in U\mid   d\mu$  has maximal rank at $x\}$.
\item    The set of smooth points $N_{\sm}$ of $N$ is $N\cap R$.
\item    The shell   is normal if and only if $ N\setminus R$ has  codimension at least two in $N$.
\item     The shell is reduced and irreducible if and only if $V$ is $1$-modular.
\item  The shell is factorial if and only if $V$ is $2$-modular.
\end{enumerate}
\end{proposition}

\begin{remark}\label{rem: change.V.UTCLS} Let $V'$ be a Lagrangian $G$-submodule of $U$. Then the proposition and Lemma \ref{lem:mu.formula} show that $V$ is $k$-modular if and only if $V'$ is $k$-modular for $k\leq 2$.   This holds for any $k\geq 0$ by   \cite[Theorem 2.4]{PanyushevJacob} which shows that $k$-modularity of $V$ is equivalent to a  homological condition on
$\C[N_V]$. However, this is not the case for the properties of stable, $k$-principal, $k$-large, or orthogonal as illustrated by
the following example. See also the examples in Section \ref{sec:Examples}.
\end{remark}

\begin{example}
\label{ex:UTCLS-circle}
Let $V$ be the $\C^\times$-module with weight vector $(1,1)$. Then as $V_{(1)} = \{0\}$, $V$ is $1$-modular.
It is clear that $V$ is not orthogonal, and as the only closed orbit in $V$ is the origin, $V$ is neither stable
nor $1$-principal. However, $V\oplus V^*$ has weight vector $(1,1,-1,-1)$ and hence is isomorphic to $V'\oplus (V')^*$
where $V'$ is a Lagrangian submodule with weight vector $(1,-1)$. As $V'$ is stable, orthogonal, and $1$-large, $V$ is
stable, orthogonal and $1$-large UTCLS.
\end{example}

In the real case (see \cite[Section 2]{HerbigSchwarzSeaton2}), changing the Lagrangian submodule can drastically change the (real) shell.

Torus actions often have shells which are
CIFR \cite[Proposition 5.4]{HerbigSchwarzSeaton2}.
Note that $N_V$ depends only on $G^0$ and if $V$ has FPIG, then it is stable.

 \begin{theorem}\label{N.good.tori}
 Let $V$ be a $G$-module with FPIG  where $G^0$ is a torus.
Then the shell $N_V\subset V\oplus V^*$ is
 CIFR.
\end{theorem}

\begin{remark}
If $G^0$ is a torus and $V$ is a $1$-modular $G^0$-module with finite kernel, then it can be shown that, UTCLS, $V$ has FPIG. 
Hence, the hypotheses of Theorem~\ref{N.good.tori} can be relaxed to assume that $G^0$ is a torus and $V$ is a $1$-modular $G$-module
with finite kernel. This will be elaborated in a forthcoming paper.
\end{remark}

\subsection{Normality and FPIG}

Let $V$ be a $G$-module and $U=V\oplus V^*$ with the canonical symplectic structure and moment mapping. 
The null cones $\NN(V)$ and $\NN(U)$ are the unions of the $G$-orbits with closure containing the origin. 
Note that
$N = N_V$, 
$\NN(V)$ and $\NN(U)$ 
only depend  upon $V$ as a $G^0$-module.
We will use the following construction and notation throughout the paper.
Let $(S,H)$ be a symplectic slice representation of $N$ with the induced symplectic form $\omega_S$. Write $S=S^H\oplus S_0$. Then $\omega_S$ is non-degenerate on $S^H$ and $S_0$. There is a Lagrangian $H$-submodule
 $W=W^H\oplus W_0$
 of $S$ where $W^H\subset S^H$ and $W_0\subset S_0$ are Lagrangian.  Let
 $N_0 = N_{W_0}$ (resp.\ $N_W$) denote the shell in $S_0$ (resp.\ $S$).  Then $N_W=S^H\times N_0$. Note that any irreducible component of $N_0$ has dimension at least $\dim S_0-\dim H$.
 Let $\NN(N_0)=\NN(S_0)\cap N_0$ and let $\NN(N_0)_\sg$ denote $\NN(N_0)\cap (N_0)_\sing$, the points in $\NN(N_0)$ which are singular points of $N_0$.

\begin{definition}\label{def:F}
We say that a
symplectic slice representations $(S,H)$  of $N$
\emph{has property (F) \/} if $H$ is finite or
 $\dim  \NN(N_0) <\dim S_0-H$.
\end{definition}

\begin{proposition}\label{prop:(F)}
The following are equivalent.
\begin{enumerate}
\item Every symplectic slice representation  of $N$ has property (F).
\item The shell $N$  is a reduced complete intersection  and each irreducible component has FPIG.
\end{enumerate}
\end{proposition}

\begin{proof}
Suppose that  (1) holds.  Let $(S,H)$ be a symplectic slice representation of $N$  where $\dim H>0$. Let   $N_\red$ denote $N$ with its reduced structure and let $\pi\colon N_\red\to Z$ denote the  quotient by $G$. By the symplectic slice theorem,  $Z_{(H)}$  and $S^H$ have the same dimension and the fibers of $\pi$ over $Z_{(H)}$ are isomorphic to
$G\times^H  \NN(N_0)$. It follows that
\begin{align*}
\dim\pi\inv(Z_{(H)})&=\dim S^H+\dim G-\dim H+\dim \NN(N_0) \\
&<\dim S+\dim G-2\dim H\\
&=2\dim V-2\dim G/H+\dim G-2\dim H \\
&=2\dim V-\dim G.
\end{align*}
Let  $N'$ be an  irreducible component $N$. Then $\dim N'\geq 2\dim V-\dim G>\dim\pi\inv(Z_{(H)})$. This inequality holds for all  strata $Z_{(H)}$ where $\dim H>0$. Thus  $N'$ contains a closed orbit $Gx$ with $G_x$ finite. It follows that $\dim_x N'=2\dim V-\dim G$.
Then by Proposition \ref{prop:props.of.N}, $x$ is a smooth point of $N'$, $N$
is a reduced complete intersection and (2) holds.

Conversely, suppose that (2) holds while  (1) fails. Then there is a stratum  $Z_{(H)}$ where  $\dim H>0$ and $\dim\pi\inv(Z_{(H)})= 2\dim V-\dim G=\dim N$. Thus $\pi\inv(Z_{(H)})$ contains an irreducible component of $N$ which does not have FPIG, a contradiction.
\end{proof}

\begin{corollary}\label{rem:0-modular}
If $N$ is a complete intersection, then every $N_W$ is a complete intersection. Equivalently, every Lagrangian $H$-submodule of $S_0$ is $0$-modular.
\end{corollary}

\begin{proof}
By Corollary \ref{cor:dim.S},
\begin{align*}
\dim N_W &=\dim N-\dim G+\dim H \\
&=  2\dim V-2\dim G+\dim H \\
&=\dim S-\dim H.
\end{align*}
Thus $N_W=S^H\times N_0$  is a complete intersection.
\end{proof}

\begin{definition}\label{def:N}
We say that a slice representations $(S,H)$   has property   (N)  if  either $H$ is finite or
$$
\dim  \NN(N_0)_\sg \leq \dim S_0-\dim H-2.
$$
\end{definition}

\begin{proposition}\label{prop:(NF)} The following are equivalent.
\begin{enumerate}
\item The shell $N$ is a normal  complete intersection with FPIG.
\item Every symplectic slice representation of $N$ has properties  (N) and (F).
\end{enumerate}
\end{proposition}

\begin{proof}
Let $(S,H)$ be a  symplectic slice representation of $N$. If $N$ has the properties in (1), then
by Corollaries \ref{cor:property(P)} and \ref{rem:0-modular}, so does $N_0$. Hence $(S,H)$ has property (F)
by Proposition \ref{prop:(F)}, and property (N) follows from normality.
Conversely, if (2) holds, then by Proposition \ref{prop:(F)}, $N$ is a reduced complete intersection each of whose irreducible components has FPIG. By induction  we may assume that
$N_0\setminus\NN(N_0)$
is normal since every closed orbit there has a symplectic slice representation $(S',H')$ where $(H')< (H)$. Then $N_0\setminus\NN(N_0)$
has singularities in codimension two and by    (N), so does $N_0$. Hence  $N_0$ is normal and by induction, $N$ is normal. Thus (1) holds.
\end{proof}

\begin{remarks}
We don't know of any examples where $N$ is normal and does not have FPIG. If $(S,H)$ is a symplectic slice representation such that, UTCLS, $S_0\simeq W_0\oplus W_0^*$ where $W_0$ is $1$-large, then $W_0$ has FPIG, hence so do  $N_W$ and $N$. If $H$ is semisimple and $W_0$ is $1$-modular, then $W_0$ is stable \cite{PopovStability,LunaVust}, hence $1$-large.
\end{remarks}

Now we introduce an important definition.
\begin{definition}\label{def:m0}
Let $V$ be a $G$-module. Define
$$
m_0(V)=\max\{\dim L \mid L\subset\NN(V), L\text{ is linear}\}.
$$
\end{definition}

\begin{remark}\label{rem:V.and.V*.same.condition}
From the proof of \cite[Lemma 3.4]{HerbigSchwarzSeaton2} we see that $m_0(V)=m_0(V^*)$ for any $G$-module $V$.
\end{remark}

\begin{lemma}\label{lem:delta}
Let $(S\simeq W\oplus W^*,H)$ be a symplectic slice representation of $N$ where $\dim H>0$. Let $U$ be a maximal unipotent subgroup of $H^0$ and let
$$
\delta= \dim W_0-\dim H-m_0(W_0).
$$
Then
\begin{enumerate}
\item $\dim\NN(W_0)\leq \dim W_0-\delta-\dim  H/U$.
\item $\dim \NN(S_0) \leq \dim S_0-2\delta-\dim H-\dim H/U$.
\end{enumerate}
\end{lemma}

\begin{proof}
Let $\lambda\colon \C^\times\to H$ be a 1-parameter subgroup. Let   $Z_\lambda(W_0)$ denote the sum of the positive weight spaces of $\lambda$ on $W_0$. Then the dimension of $\NN(W_0)$ is at most the maximum over $\lambda$ of  $\dim Z_\lambda(W_0)+\dim U$.
Thus as each $\dim Z_\lambda(W_0)\leq m_0(W_0)$,
$$
\dim\NN(W_0) \leq \dim W_0-\dim H-\delta+\dim U=\dim W_0-\delta-\dim H/U
$$
which is (1). Since $\dim Z_\lambda(S_0)\leq 2m_0(W_0)$, we similarly get that
$$
\dim\NN(S_0)\leq \dim S_0-2\dim H-2\delta+\dim U
$$
giving (2).
\end{proof}

Using Lemma \ref{lem:delta}, Theorem \ref{N.good.tori} and Propositions \ref{prop:props.of.N} and \ref{prop:(NF)} we obtain the following.

\begin{proposition}\label{prop:N.and.F}
Assume that for every symplectic slice representation $(S,H)$ of $N$ with $\dim H>0$, one of the following
holds (possibly UTCLS).
 \begin{enumerate}
\item $N_0$ is a normal complete intersection
with FPIG.
\item $H^0$ is a torus and $W_0$
has FPIG.
\item Properties (F) and (N).
\item $m_0(W_0)< \dim W_0-\dim H$.
\end{enumerate}
Then  each $N_0$ is a normal complete intersection with FPIG and each Lagrangian $H$-submodule of $S_0$ is   $1$-modular. In particular, $N$ is a normal complete intersection with FPIG and   $V$ is $1$-modular.
\end{proposition}


\subsection{Musta\c{t}\u{a}'s criterion for rational singularities}
\label{subsec:BackReps}

In this section, varieties will be assumed to be reduced and irreducible.
Recall that a complex variety $X$ has rational singularities   if $X$ is  normal    and  for every resolution
of singularities $f\co Y\to X$, $R^i f_\ast \calO_Y = 0$ for $i > 0$; see \cite[p.\ 188]{Kovacs}.
This condition is local (and open), so one can talk about $X$ having a rational singularity at $x\in X$.
We let $X_\sm$ (resp.\ $X_\sing$) denote the smooth (resp.\ singular) points of $X$.

Let $f\in\C[x_1,\dots,x_n]$. For any $m\geq 1$, define new variables $x_i^j$, $j=1,\dots,m$, $i=1,\dots, n$. For a variable $t$, let $x_i(t)=x_i+tx_i^1+\dots+t^mx_i^m$ and   let $f=0,f_{(1)}=0,\dots, f_{(m)}=0$  be the  
coefficients of $t^i$, $i=0,\dots,m$, in the equation 
   $f(x_1(t),\dots,x_n(t))=0\mod t^{m+1}$. If $X\subset \C^n$ is an affine variety defined by an ideal $(f^1,\dots,f^k)$, then the \emph{$m$th jet scheme $X_m$ of $X$\/} is defined by the ideal of the polynomials  $f^1,\dots,f^k,f^1_{(1)},\dots,f^k_{(1)},\dots,  f^1_{(m)},\dots, f^k_{(m)}$. There is a natural projection $\rho_m\colon X_m\to X$. This construction is local, so one can define $X_m$ and $\rho\colon X_m\to X$   if $X$ is a variety that is not affine.

The main criteria for rational singularities we use is the following result of Musta\c{t}\u{a}.
\begin{theorem}[{\cite[Theorem~0.1 and Proposition~1.4]{MustataJetRational}}]
\label{thrm:MustataGen}
Let $X$ be a local complete intersection variety over an algebraically closed field of characteristic $0$ and let $m\geq 1$. 
\begin{enumerate}
\item The closure of $\rho_m\inv(X_\sm)$ is an irreducible component of $X_m$ of dimension $(m+1)\dim X$.
\end{enumerate}
The following are equivalent.
\begin{enumerate}
\addtocounter{enumi}{1}
\item      $X_m$ is irreducible.
\item      $\dim\rho_m^{-1}(X_{\sing}) < (m+1)\dim X$.
\end{enumerate}
The equivalent conditions (2) and (3) imply that $\dim X_m=(m+1)\dim X$.
Moreover, $X$ has rational singularities if and only if (2) and (3) hold for every $m\geq 1$.
\end{theorem}

Let $V$ be a $G$-module. Then by Proposition~\ref{prop:props.of.N}(5), $N$ is a reduced and irreducible complete intersection if and only if $V$ is $1$-modular, which we now assume. Hence we can apply
Theorem~\ref{thrm:MustataGen},  which we restate in this context as follows.

\begin{proposition}
\label{prop:Mustata}
Let $G$ be a complex reductive   group, $V$ a $1$-modular $G$-module and
$N = N_V\subset V\oplus V^\ast$ the shell.
The following are equivalent for $m\geq 1$.
\begin{enumerate}
\item   $N_m$ is irreducible.
\item   $\dim\rho_m\inv(N_\sing)<  (m+1)\dim N$.
\end{enumerate}
The equivalent conditions (1) and (2) imply that $\dim N_m=(m+1)\dim N$.
Moreover, $N$ has rational singularities if and only if (1) and (2) hold for all $m\geq 1$.
\end{proposition}

\begin{definition}
Let $V$ be a $1$-modular $G$-module
and $N=N_V$.   We say that \emph{$\NN(N)$ is irrelevant} if
$N_m$ is the closure of $\rho_m\inv(N\setminus \NN(N)_\sg)$ for any $m\geq 1$.
\end{definition}

\begin{proposition}\label{prop:irrelevant} Let $V$ be a $1$-modular $G$-module and   $N=N_V$.
 If $N\setminus\NN(N)$ has rational singularities and $\NN(N)$ is irrelevant, then $N$ has rational singularities.
\end{proposition}

\begin{proof}
Let $m\geq 1$ and let $N'=N\setminus\NN(N)_\sg$. By hypothesis, $N'$ has rational singularities and $N_m$ is the closure of $\rho_m\inv(N')$  where the latter is irreducible of dimension $(m+1)\dim N$. Hence $N$ has rational singularities.
\end{proof}

Recall that the shell $N$ is CIFR if it is a complete intersection with FPIG and rational singularities.

\begin{theorem}\label{thm:iff}
Let $V$ be a  $G$-module and $N=N_V$. We assume that $N$ is a normal complete intersection with FPIG.  Then $N$ is CIFR   if and only if $\NN(N_0)$ is irrelevant for every symplectic slice representation $(S=S^H\oplus S_0,H)$ of $N$.
\end{theorem}

\begin{proof}
Let $(S=S^H\oplus S_0,H)$ be a symplectic slice representation of $N$ and
 let $\rho_m\colon (N_0)_m\to N_0$ denote the projection.  By Proposition  \ref{prop:(NF)},   $N_0$ is a normal complete intersection with FPIG. If $N$ has rational singularities, then so does $N_0$
by Corollary \ref{cor:P}
 and we may apply Proposition \ref{prop:Mustata} which gives that
$$
\dim\rho_m\inv((N_0)_\sing)<(m+1)\dim N_0.
$$
Since any irreducible component of $(N_0)_m$ has dimension at least $(m+1)\dim N_0$, the above shows that    $\NN(N_0)$ is  irrelevant.

Conversely, suppose that each $\NN(N_0)$ is irrelevant. Given a particular $(S,H)$, we may assume by induction that for any symplectic slice representation $(S',H')$ with $(H')<(H)$, the shell $N_{0}'$ has rational singularities. (The induction starts with the case that $H$ is a principal isotropy group in which case $N_0$ is smooth.)\ By Corollary \ref{cor:property(P)} and Remark \ref{rem:closed.orbits}, $N_0\setminus \NN(N_0)$  has rational singularities.
By Proposition \ref{prop:irrelevant},   $N_0$ has rational singularities. It follows that  $N_0$  is CIFR, hence $N$ is CIFR.
\end{proof}


\section{The jet schemes of the shell}
\label{sec:JetShell}

Let $G$ be a reductive complex group and $V$ a $1$-modular $G$-module. Let  $N=N_V$ be the shell and let $m\geq 1$.
Then   $N_m$  is a subscheme of $V^{m+1}\oplus (V^\ast)^{m+1}$. As explained in Section~\ref{subsec:BackReps}, $N_m$ 
is the subscheme defined by the following system of equations 
where
$(x_0,\ldots,x_m)\in V^{m+1}$, $(\xi_0,\ldots,\xi_m)\in (V^\ast)^{m+1}$, and $A$ runs through a basis of $\lieg$.
\begin{align}
    \xi_0(A(x_0))
            &=      0,\tag {3.0}  \label{eq:A1}
    \\
    \xi_0(A(x_1)) + \xi_1(A(x_0))
            &=      0,\tag {3.1}\label{eq:A2}
   \\
    \xi_0(A(x_2)) + \xi_1(A(x_1)) + \xi_2(A(x_0))
            &=      0,\tag {3.2}\label{eq:A3}
   \\
  \xi_0(A(x_3))+\xi_1(A(x_2))+\xi_2(A(x_1))+\xi_3(A(x_0))
  &= 0,\tag {3.3}\label{eq:A4}
    \\\vdots&\nonumber \\
   \xi_0( A(x_m)) +  \xi_1(A(x_{m-1}))  +  \cdots   +   \xi_{m-1}(A(x_1)) + \xi_m(A(x_0))
            &=      0\tag{3.m}\label{eq:Am}.
\end{align}

\begin{remark}\label{rem:Nm.ci}
If $N$ has rational singularities, then each $N_m$ has dimension $(m+1)\dim N$ and hence by counting equations is a complete intersection. Moreover, each $N_m$ is a variety \cite[Proposition~1.5]{MustataJetRational}.
\end{remark}

We observe the following consequences of the description of $N_m$ above.

\begin{lemma}
\label{lem:r-1Necessary}
The equivalent conditions (1) and (2) of Proposition~\ref{prop:Mustata} hold with
$m = 1$ if and only if $N$ is $1$-modular as a $G$-variety.
\end{lemma}
\begin{proof}
By Proposition \ref{prop:props.of.N}(3), $N_{\sing} = \bigcup_{r=1}^{\dim G} N_{(r)}$.
If $(x_0,\xi_0) \in N_{(r)}$ with $r \geq 1$, then $\rho_1^{-1}(x_0,\xi_0)$ is defined by the linear system
\eqref{eq:A2} in $(x_1,\xi_1)$, which has rank $\dim G - r$. Therefore, $\rho_1^{-1}(N_{(r)})$ is a vector bundle
over $N_{(r)}$ of rank $2\dim V - \dim G + r = \dim N + r$, and
\[
    \dim \rho_1^{-1}(N_{(r)}) = \dim N_{(r)} + \dim N + r.
\]
Hence $\dim \rho_1^{-1}(N_{(r)}) < 2\dim N$ if and only if $\dim N - \dim N_{(r)} \geq r + 1$, which holds for each
$r \geq 1$ if and only if $N$ is $1$-modular.
\end{proof}


\subsection{The shell has rational singularities for generic $V$}
\label{subsec:JetShellGeneric}

Let $\pi\colon N\to N\git G$ denote the categorical quotient.
The next several  results
follow from the techniques in Budur \cite{BudurRational}.

\begin{lemma}\label{lem:fiber.dim}
Let $V$ be a $1$-modular $G$-module and $N=N_V$ Then for any $m\geq 1$ and $(x_0,\xi_0)\in N$,
$\dim\rho_m\inv(x_0,\xi_0)\leq\dim\rho_m\inv(0,0)$. If $m=1$, then
 $\rho_1\inv(0,0)\simeq V\oplus V^*$ and if $m\geq 2$, then
$$
\rho_m\inv(0,0)\simeq N_{m-2}\times V\times V^*.
$$
\end{lemma}
\begin{proof}
Since $N_m$ and $N$ are cones, the fiber $\rho_m\inv(x_0,\xi_0)$ has dimension at most that of $\rho_m\inv(0,0)$. The latter is given by equations
$$
\sum_{i+j=k}\xi_iA(x_j)=0,\quad A\in\lieg,\ i,\ j\geq 1,\ k=2,\dots,m-1,
$$
where $x_1,\dots,x_m\in V$ and $\xi_1,\dots,\xi_m\in V^*$. Note that there are no conditions on $x_m$ and $\xi_m$ and that the  equations on the remaining variables define a copy of $N_{m-2}$.
\end{proof}

\begin{theorem}\label{thm:Budur3}
Let $V$ be $1$-modular and $N=N_V$. Let $N'\subset N$ be a closed subvariety such that $N\setminus N'$ has rational singularities. If $\codim_N (N'\cap N_\sing)>\dim G$, then  
each $N_m$ is irreducible, hence $N$ has rational singularities.
\end{theorem}

\begin{proof}
We may replace $N'$ by $N'\cap N_\sing$. We leave it to the reader to show that $N_1$ is irreducible, so assume that $m\geq 2$. By induction we may assume that $N_{m-2}$  is irreducible of dimension $(m-1)\dim N$.  Then
$$
\dim\rho_m\inv(N_\sing\setminus N')<(m+1)\dim N,\text{ and }
$$
\begin{align*}
\dim\rho_m\inv(N')&\leq \dim N'+\dim N_{m-2}+2\dim V\\
&<\dim N-\dim G+(m-1)\dim N+2\dim V=(m+1)\dim N.
\end{align*}
It follows that  $N_m$ is irreducible. Then $N$ has rational singularities by Proposition \ref{prop:Mustata}.
\end{proof}

\begin{corollary}\label{cor:Budur}
Let $V$ be $1$-modular and $N=N_V$. If $\codim_NN_\sing>\dim G$, then $N$ has rational singularities.
\end{corollary}

\begin{corollary}\label{cor:NN(N).irrelevant}
Let $V$ be $1$-modular and $N=N_V$. Suppose that $N\setminus\NN(N)$ has rational singularities and $\codim_N \NN(N)_\sg>\dim G$. Then 
$\NN(N)$ 
is irrelevant and $N$ has rational singularities.
\end{corollary}

\begin{corollary} \label{cor:dimGMod}
If $V$ is $(\dim G)$-modular, then $N=N_V$ has rational singularities.
\end{corollary}

\begin{proof}
Note that $N_\sing$ is the union of the $N_{(r)}$ for  $r\geq 1$.  Now
$$
N_{(r)}\subset \bigcup_{s,t\geq r} V_{(s)}\times V^*_{(t)}
$$
and it follows that
$$
\dim  N_{(r)}\leq 2\dim V-2r-2\dim G=\dim N-2r-\dim G
$$
and Corollary \ref{cor:Budur} applies.
\end{proof}

\begin{remark}
\label{rem:dimGModGeneric}
By \cite[Theorem~3.6]{HerbigSchwarzSeaton2}, if $G$
is semisimple, then among $G$-modules $V$ such that $V^G = \{0\}$ and such that each irreducible
component of $V$ is a faithful $\mathfrak{g}$-module, all but finitely many are $k$-modular, up to isomorphism, for
any $k$. Hence, Corollary~\ref{cor:dimGMod} demonstrates that among such $G$-modules, the shell has
rational singularities in all but finitely many cases.
\end{remark}

\subsection{Linear subspaces of the null cone and rational singularities}
\label{subsec:JetShellLinSubspace}

Let $V$ be a
$G$-module and $N=N_V$.  We assume that $V^G=0$. Some of our criteria above for $N$ to be CIFR depended upon estimating the
 dimensions of $\NN(N_0)_\sg$ for symplectic slice representations $(S=S^H\oplus S_0,H)$ of $N$. We develop another approach which relies upon estimates of $m_0(W_0^*)$ for $W_0$ a Lagrangian $H$-submodule of $S_0$.

Let $n=\dim V$. Fix $\vec x=(x_0,\dots,x_m)\in V^{m+1}$ and let
\begin{equation}
\label{eq:Yx}\tag{$\diamondsuit$}
Y_{\vec x}=\{\vec\xi\in (V^*)^{m+1}\mid (\vec x,\vec \xi)\in N_m\}.
\end{equation}
Note that $Y_{\vec x}$ is defined by linear equations.

Let $E_0$ denote $\lieg(x_0)$ and let $\lieg_0$ denote $\Lie(G_{x_0})$.
For $i=0,\dots,m-1$, inductively define the subspace $E_{i+1}=\lieg_i(x_{i+1})+E_i$ and the Lie subalgebra $\lieg_{i+1}$ as the set of $A \in \lieg_i$ such that $A(x_{i+1})\in E_i$, i.e., the kernel of the map $\lieg_i\to V/E_i$ sending $A$ to $A(x_{i+1})+E_i$.
Note that each $E_i$ is $\lieg_i$-stable.  Choose linear subspaces $\lie p_0,\dots,\lie p_{m}$ of $\lieg$ such that $\lieg=\lieg_0\oplus\lie p_0$ and $\lieg_i=\lieg_{i+1}\oplus \lie p_{i+1}$ for $0\leq i<m$. Set $E_i'=\lie p_i(x_i)$, $0<i\leq m$. Then $E_i=E_{i-1}\oplus E_i'$, $0< i\leq m.$ Let $r_i=\dim E_i$, $i=0,\dots,m$.

\begin{lemma}\label{lem:fibers.Y.to.X}
Let $\vec\xi\in(V^*)^{m+1}$. Then $\vec\xi\in Y_{\vec x}$ if and only if
\begin{enumerate}
\item $\xi_0\in\Ann E_m$.
\item For $0<i\leq m$, $\xi_i$ restricted to $E_{m-i}$ satisfies linear equations
of the form $\xi_i(A(x_{m-i}))= C_{m,i}$ where the right hand sides $C_{m,i}$ are determined by $\xi_0,\dots,\xi_{i-1}$.
\end{enumerate}
\end{lemma}

\begin{proof}
The proof proceeds by induction on $m$. For $m=0$, Equation \eqref{eq:A1} gives (1) and part (2) is vacuous. Assume the lemma holds for $m-1$. Now $\lieg=\lieg_m\oplus \liep_m\oplus\liep_{m-1}\oplus\dots\oplus\liep_0$ where for $A\in\lieg_m$,  Equation \eqref{eq:Am} vanishes.
For $A\in\liep_m$,  Equation \eqref{eq:Am} shows that $\xi_0$ vanishes on $E_m'$. Since we already know that $\xi_0$ vanishes on $E_{m -1}$, it follows that $\xi_0\in\Ann E_m$ and we have (1). For $0< i\leq m$, Equation \eqref{eq:Am} with $A\in\liep_{m-i}$ becomes
$$
\xi_i(A(x_{m-i}))=-\xi_{i-1}(A(x_{m-i+1}))-\dots-\xi_0(A(x_m)).
$$
Thus $\xi_i$ restricted to $E_{m-i}'$ is uniquely determined by $\xi_0,\dots,\xi_{i-1}$. Since the same is true for $\xi_i$ restricted to $E_{m-i-1}$ we have  (2).
\end{proof}

Let $\tau\colon Y_{\vec x}\to V^*$  be the projection sending $\vec\xi$ to $\xi_0$.
\begin{corollary}\label{cor:tau.on.Y}
With $Y_{\vec x}$ defined as in Equation \eqref{eq:Yx}, we have the following.
\begin{enumerate}
\item The dimension of $Y_{\vec x}$ is $(m+1)n-\sum_{i=0}^m r_i$.
\item The projection $\tau$ has image $\Ann E_m$. Its    fibers are affine subspaces of dimension $mn-\sum_{i=0}^{m-1} r_i$.
\end{enumerate}
\end{corollary}

Let $R_m=\{\vec r=(r_0,\dots,r_m) \in \Z^{m+1}\mid 0\leq r_0\leq r_1\dots\leq r_m\leq\dim G\}$. For each $\vec r \in R_m$, let  $X_{\vec r}$ be the  set of $\vec x\in V^{m+1}$   such that $\dim E_i=r_i$, $i=0,\dots,m$.
Let $X$ be an irreducible component of $X_{\vec r}$ and let $Y$ denote the solutions to \eqref{eq:A1}--\eqref{eq:Am} with $\vec x \in X$.  By
Corollary \ref{cor:tau.on.Y},
$Y$ is irreducible of dimension $\dim X+(m+1)n-\sum_{i=0}^{m} r_i$.
If $r_0=\dim G$, then for any $(\vec x,\vec \xi)\in Y$, the isotropy group of  $(x_0,\xi_0)$ is finite, so by  Proposition \ref{prop:props.of.N}(3),   $(x_0,\xi_0)\in N_\sm$.

\begin{corollary}
\label{cor:rat.sings.by.codim}
Let $V$ be a $1$-modular $G$-module and $N=N_V$. Suppose that, UTCLS, for each $\vec r\in R_m$   and each
irreducible component $X$ of $X_{\vec r}$ where $r_0<\dim G$, the codimension of $X$ in $V^{ m+1}$ is greater than
$(m+1)\dim G-\sum_{i=0}^{m} r_i$. Then  $N$ has rational singularities.
\end{corollary}

\begin{proof}
By
Corollary \ref{cor:tau.on.Y}(1)
and our hypothesis, $\dim Y<(m+1)\dim N$. It follows that
$$
\dim\rho_m\inv(N_\sing)<(m+1)\dim N,
$$
hence $N$ has rational singularities
by Proposition \ref{prop:Mustata}.
\end{proof}

Recall the definition of $m_0(V)$ (Definition \ref{def:m0}).

\begin{corollary}\label{cor:dim.L}
Let
 $V$ be a $1$-modular $G$-module with $V^G=0$. Suppose that,
 UTCLS,
 $$m_0(V)<\dim V-\dim G.
 $$
  Then $\NN(N)$ is irrelevant.
  \end{corollary}

 \begin{proof}  Let $N_m'$ be the closure of $\rho_m\inv(N\setminus\NN(N)_\sg)$.
Let $\vec x\in V^{m+1}$ and let   $Y_{\vec x}$ be
as in Equation~\eqref{eq:Yx}.
Since $\dim E_m\leq\dim G$,
$$
 \dim\Ann E_m\geq \dim V-\dim G>m_0(V^*)
$$
 and   $Y_{\vec x}$ contains a dense open set of points  $\vec\xi$ where  $ \xi_0\not\in\NN(V^*)$, hence $(x_0,\xi_0)\not\in\NN(N)$. Thus $(\{\vec x\}\times Y_{\vec x})\cap N_m'$ is  dense in $\{\vec x\}\times Y_{\vec x}$, hence $\{\vec x\}\times Y_{\vec x}\subset N_m'$. Thus $N_m=N_m'$ and $\NN(N)$ is irrelevant.
 \end{proof}

\begin{remark}\label{rem:Em.not.in.nullcone}
The argument
in the proof of Corollary \ref{cor:dim.L}
shows that $\NN(N)$ is irrelevant  if we can find any reason that no $\Ann E_m$ is  contained in $\NN(V^*)$.
\end{remark}

\begin{theorem}\label{thm:use.E_m}
Let $V$ be a   $G$-module and $N=N_V$.
Suppose that for any symplectic slice representation $(S=W\oplus W^*,H)$ of $N$, 
where $\dim H>0$,
one of the following holds, UTCLS.
\begin{enumerate}
\item $N_0$ is CIFR.
\item $H^0$ is a torus and $W_0$ has FPIG.
\item $\dim (\NN(N_0)_\sg)<\dim S_0-2\dim H$ and $\dim (\NN(N_0))<\dim S_0-H$.
\item  $m_0(W_0)<\dim W_0-\dim H$.
\end{enumerate}
Then $N$ is CIFR.
 \end{theorem}

\begin{proof}
By Proposition \ref{prop:N.and.F}, $N$ is a normal complete intersection with FPIG. By Theorem \ref{thm:iff} we need only show that each $\NN(N_0)$ is irrelevant. This is clear in (1), (2) and (4), and in (3) we may assume by induction that $N_0\setminus\NN(N_0)$ has rational singularities and apply Corollary \ref{cor:NN(N).irrelevant}.
\end{proof}

We can say more when $V$ is an orthogonal $G$-module, i.e., if $V$ admits a non-degenerate symmetric $G$-invariant bilinear form. Let  $K$ be a  maximal compact subgroup of $G$ so that $G=K_\C$ is the complexification of $K$.  Then $V=W\otimes_\R\C$ where $W$ is a real $K$-module \cite[Prop.\ 5.7]{GWSliftingHomotopies}.
Let  $T$  be the complexification of a maximal torus $T_0$ of $K$. Let $\Lambda$ be the nonzero weights of $V$ relative to $T$. Then $\Lambda=-\Lambda$. If $\lambda\in\Lambda$, let $V_\lambda$ denote the corresponding weight space of $V$. Since $V$ is orthogonal, $\dim V_\lambda=\dim V_{-\lambda}$.

\begin{proposition}\label{prop:dim.L} If $V$ is orthogonal, then $m_0(V)= (1/2)(\dim V-\dim V^T)$.
\end{proposition}

\begin{proof}
Let $\mu_0(V)=(1/2)(\dim V-\dim V^T)$ and
let $L\subset\NN(V)$ be linear of dimension $m$. First note that if $\NN(V)$ contains a linear subspace of dimension $m$, then $\NN(V)$ contains a linear $T$-stable subspace of dimension $m$. This is shown using an argument from the proof of \cite[Lemma 1]{Draisma-Kraft-Kuttler}, which we now recall. Let  $\mathrm{Gr}_m(V)$ denote the Grassman variety of $m$-dimensional subspaces of $V$. Let $Z_m$ denote the elements of $\mathrm{Gr}_m(V)$ which lie in $\NN(V)$. Then $Z_m$ is a nonempty closed $G$-stable subvariety of $\mathrm{Gr}_m(V)$. By Borel's fixed point theorem \cite[III.10.4]{Borel}, $Z_m$ contains a $T$-fixed point. Thus we may assume that $L$ is   $T$-stable so that $L$ is the direct sum of weight spaces $L_\lambda,\ \lambda\in\Lambda$, although some $L_\lambda$ may be zero.

 Now each $V_\lambda\oplus V_{-\lambda}$ is $W_\lambda'\otimes_\R\C$ where $W_\lambda'$ is a $T_0$-stable subspace of $W$. Suppose that $m> \mu_0(V)$. Then for some $\lambda\in\Lambda$, $\dim L_\lambda+\dim L_{-\lambda}>\dim V_\lambda$ and
 $$
\codim_{(V_\lambda+V_{-\lambda})} (L_\lambda+L_{-\lambda})<\dim V_\lambda.
$$
The real dimension of $W_\lambda'$ is $2\dim V_\lambda$. Thus $W_\lambda'$ and $L$ have a positive dimensional intersection.  But all points of  $W_\lambda'$ lie on closed $K$-orbits, hence on closed $G$-orbits \cite{Birkes}. Thus $L$ is not contained in $\NN(V)$, a contradiction. Hence $m_0(V)\leq \mu_0(V)$.

Let
$\rho\colon \C^\times \to T$ be a  1-parameter subgroup which acts nontrivially on every $V_\lambda$, $\lambda\neq 0$.  Let
$$
L=\Span\{ V_\lambda\mid \rho\text{ has strictly positive weight on }V_\lambda\}.
$$
 Then $L\subset\NN(V)$ and $\dim L=(1/2)(\dim V-\dim V^T)=\mu_0(V)$. Thus $m_0(V)=\mu_0(V)$.
 \end{proof}

 As far as we know, the question of the dimension of linear subspaces of $\NN(V)$ has only been investigated when $G$ is
 reductive
 and $V=\lieg$.  Proposition \ref{prop:dim.L} was  then established by Gerstenhaber \cite{Gerstenhaber} for $G=\SL_n(\C)$ and for general semisimple $G$ by  Meshulam  and  Radwan \cite{Meshulam-Radwan}. See also \cite{Draisma-Kraft-Kuttler} by Draisma, Kraft and Kuttler. These works also consider more general fields than $\C$ and the question of conjugacy of maximal dimensional linear subspaces of $\NN(\lieg)$.

 \begin{lemma}\label{lem:slice.orthog.rep}
 Let $V$ be an orthogonal $G$-module and  let $v\in V$ lie on a closed orbit with isotropy group $H$.
 \begin{enumerate}
\item The slice representation $(W,H)$ is orthogonal.
\item  The symplectic slice representation at $(v,0)\in N\subset V\oplus V^*$  is $(W\oplus W^*,H)$.
\end{enumerate}
 \end{lemma}

 \begin{proof}
Let $K$ be a maximal compact subgroup of $G$. There is a real $K$-module $U$ such that $V\simeq U\otimes_\R\C$.  By \cite[Prop.\ 5.8]{GWSliftingHomotopies}, the isotropy groups of closed
$G$-orbits
in $V$ are conjugate to the complexifications of those occurring in the $K$-module $U$. Hence $W$ is orthogonal. As 
an $H$-module, $V\simeq W\oplus (\lieg/\lie h)$ and
$$
V\oplus V^*\simeq W\oplus W^*\oplus \lieg/\lie h\oplus (\lieg/\lie h)^*.
$$
Thus the symplectic slice representation at $(v,0)$ is $(W\oplus W^*,H)$.
 \end{proof}

 \begin{proposition}\label{prop:orthog.reps}
Assume that $V$ is orthogonal and let $N\subset V\oplus V^*$ be the  shell. Then the symplectic slice representations   are precisely
those of the form $(W\oplus W^*,H)$ where $(W,H)$ is a slice representation of $V$.
\end{proposition}

\begin{proof}
Let $K$ and $U$ be as in the proof of Lemma \ref{lem:slice.orthog.rep}.   Then $V\oplus V^*\simeq V\oplus V$ is the complexification of $U\oplus U$. The isotropy groups of $K$ which occur in $U\oplus U$ are all in slice representations of the isotropy groups which occur in a single copy of $U$. Thus any symplectic slice representation of $V\oplus V^*$ is in turn a symplectic slice representation of one at a point $(v,0)$ or $(0,v^*)$.  By   Lemma \ref{lem:slice.orthog.rep}  the latter symplectic slice representations are of the form $(W\oplus W^*,H)$ where $(W,H)$ is a slice representation of $V$.  By induction, the proposition holds for  $(W,H)$. Since any slice representation of $(W,H)$ is also a slice representation of $(V,G)$, the proposition follows.
\end{proof}

\begin{theorem}\label{thm:rational.sings.orthogonal}
Let $V$ be a $G$-module and $N=N_V$.
Suppose that for any symplectic slice representation $(S=W\oplus W^*,H)$ of $N$, 
 where $\dim H>0$ and $T$ is a maximal torus of $H$,
one of the following holds, UTCLS.
\begin{enumerate}
\item $N_0$ is CIFR.
\item $H^0$ is a torus and $W_0$
has FPIG.
\item $\dim (\NN(N_0)_\sg)<\dim S_0-2\dim H$ and
$\dim \NN(N_0)<\dim S_0-H$.
\item $S_0$ has an orthogonal Lagrangian $H$-submodule $W_0$ and $\dim H<\frac 12(\dim W_0+\dim  W_0^T)$.
\end{enumerate}
Then $N$ is CIFR.
\end{theorem}

\begin{proof}
Clearly we only have to consider (4).   We may assume  that there is a $W_0$ as described. By Proposition \ref{prop:dim.L},
$m_0(W_0)=(1/2)(\dim W_0-\dim W_0^T)$.
Then
\begin{align*}
\dim W_0-\dim H &> \dim W_0-\frac 12 (\dim W_0+\dim W_0^T) \\
&=\frac 12(\dim W_0-\dim W_0^T) \\
&= m_0(W_0)
\end{align*}
which is (4) of Theorem \ref{thm:use.E_m}. Hence $N$ is CIFR.
\end{proof}


\section{Copies of the adjoint representation}
\label{sec:Adjoint}

In this section, we consider the case that $V=p\,\lieg$ with $p > 1$, i.e., $V$ is given by copies of the adjoint representation.
We first need the following preliminaries.
Recall that a group acts \emph{almost faithfully} if the kernel of the action is finite. 

\begin{lemma}\label{lem:2-large}\label{lem:tori.2.modular}\label{lem:codim.1.slice}
Let $V$ be an orthogonal $G$-module.
\begin{enumerate}
\item $V$ is $2$-principal if and only if $V\git G$ has no codimension one strata.
\item A slice representation $(W,H)$ corresponds to a codimension one stratum if and only if $\dim W_0\git H=1$ where $W_0$ is an $H$-module and $W^H\oplus W_0=W$.
\end{enumerate}
Now suppose that  $G=T$  is a torus and $T$ acts almost faithfully on $V$.
\begin{enumerate}
\addtocounter{enumi}{2}
\item The module  $V$ is $k$-principal if and only if it is $k$-large, $k\geq 1$.
\item If $(W,H)$ is the slice representation of a codimension one stratum, then $\dim H=1$ and $\dim W_0=2$.
\end{enumerate}
\end{lemma}

\begin{proof}
By \cite[Corollary 7.4]{GWSliftingHomotopies}, any stratum of $V\git G$ of codimension at least $2$ has inverse image in $V$ of codimension at least $2$ and (1) follows. Part (2) is obvious and (3) is \cite[Proposition 10.1]{GWSlifting}. Let $(W_0,H)$ be as in (4).  Then $\dim W_0\git H=1$ and $W_0$ is just the nontrivial part of $V$ as 
an $H$-module. By Lemma \ref{lem:slice.orthog.rep}(1), $(W_0,H)$ is orthogonal. Since the weights of $V$ occur in pairs $\pm \nu$, if $H$ is finite, then $\dim W_0\geq 2$ and $\dim W_0/H\geq 2$, a contradiction.
Since $(W_0,H)$ is almost faithful,
 $\dim W_0\geq 2\dim H$ so that $\dim W_0\git H\geq\dim H$. Hence $\dim H=1$ and $\dim W_0=2$.
\end{proof}

\begin{proposition}\label{prop:torus.on.lieg}
Let $G$ be  simple of rank at least $2$.  Let $T$ be a maximal torus of $G$ and $\Phi$ the corresponding set of roots of $\lieg$. Let $T$ act  on the span $V$ of the roots spaces  $\lieg_\alpha$, $\alpha\in\Phi$. 
Then $V$ is a $2$-large $T$-module.
\end{proposition}

\begin{proof}
If $V$ is not $2$-principal, then there is a subgroup $H$   of  $T$ of dimension $1$ with slice representation $W$ such that $\dim W_0=2$. Let $\alpha_1,\dots,\alpha_\ell$ be the simple roots of $\lieg$. Since $\dim W_0=2$, at most one $\alpha_i$ does not vanish on $\lieh$, say $\alpha_1$. Since $\Rank G >1$, there is a positive root $\alpha=\sum_i n_i\alpha_i$ where $n_1>0$ and $n_j> 0$ for  some $j> 1$. Then $\alpha(\lieh)\neq 0$ which implies that $\dim W_0\geq 4$, a contradiction.
\end{proof}

\begin{corollary}\label{cor:torus.codim.sings}
Let $(V,T)$ be as in Proposition \ref{prop:torus.on.lieg}. Let $N\subset V\oplus V^*$ be the shell. Then $N$  is
CIFR. Moreover,
for any slice representation $(W,H)$ of $V$ where $\dim H>0$, set $S=W\oplus W^*$ and  we have
$$
 \codim_{N_0}(\NN(N_0)_\sg)\geq 4.
$$
\end{corollary}

\begin{proof}
We already know from Theorem \ref{N.good.tori} that every $N_0$ is
CIFR. Let $\NN_0=\NN(S_0)$.  Then $\NN_0$ is a (finite) union of linear subspaces $Z_\lambda(S_0)$ where  $\lambda\colon \C^\times\to H$ is a 1-parameter subgroup.    Since $H$ is abelian, each $Z_\lambda(W_0)$ is an $H$-module. A given $Z_\lambda$ is  maximal (for set inclusion) if and only if $W_0^\lambda=0$, so we only need to consider such ``generic'' $\lambda$. Then $W_0\simeq Z_\lambda(W_0)\oplus Z_\lambda(W_0)^*$. Since $W_0$ is $2$-modular, as in Example \ref{ex:UTCLS-circle}, Lemma \ref{lem:mu.formula} tells use that $Z_\lambda(W_0)\oplus Z_\lambda(W_0)$, which is an irreducible component of $\NN_0$,  is $2$-modular. Thus $(\NN_0)_{(r)}$ has codimension at least $3$ in $\NN_0$ for $r>0$.
Hence
\begin{align*}
\codim_{N_0} \NN(N_0)_\sg  & \geq\dim N_0-\dim\NN_0+3 \\
&  = 2\dim W_0-\dim H-\dim W_0+3\\
&  =\dim W_0-\dim H+3\geq 4.\qedhere
\end{align*}
\end{proof}

\begin{lemma}\label{lem:lie.alg.main}
Let $G$ be simple and $V=p\,\lieg$ for $p>1$. Let $N$ denote the shell of $V\oplus V^*$ and let $(S,H)$ be  a symplectic slice representation of $N$ with $\dim H>0$ and orthogonal Lagrangian submodule $W=W^H\oplus W_0$ of $S$. Consider the  conditions
\begin{enumerate}
\item $m_0(W_0)<\dim W_0-\dim H$.
\item $H^0$ is a torus and $W_0$ has FPIG (hence $N_0$ is
CIFR).
\item   $\codim_{N_0} \NN(N_0)_\sg\geq 4$.
\item $\codim_{W_0}\NN(W_0)\geq 2$.
\end{enumerate}
Then (1) or (2) always hold.  Parts (3) and (4) fail  if and only if  $p=2$ and $\Rank G=1$.
\end{lemma}

\begin{proof}
It is easy to see that $V$ has FPIG so that all $(W_0,H)$ have FPIG.
By Proposition \ref{prop:orthog.reps}, $(W_0\oplus W^H,H)$ is a slice representation of $V$. As 
an $H$-module, $V=p\,\lieh\oplus p(\lieg/\lieh)$, hence
$$
W\simeq p\,\lieh\oplus (p-1)(\lieg/\lieh).
$$
Write $H^0=ZH_s$ where $H_s$ is semisimple and $Z$ is a central torus. Then $Z$ acts trivially on $\lieh=\lieh_s\oplus\lie z$ and almost faithfully and orthogonally on $\lieg/\lieh$. Let $\lie m$ denote the $H$-complement to $(\lieg/\lieh)^H$ in $(\lieg/\lieh)$. Then $W_0=p\,\lieh_s\oplus (p-1)\,\lie m$ where $\dim\lie m\geq 2\dim Z$.
Let $B=TU$ be a Borel subgroup of $H_s$ and let $\lie t$ and $\lie u$ be the corresponding subalgebras of $\lieh_s$ where $\dim \lie t=\ell$. Since $W_0$ is orthogonal,
$$
m_0(W_0)\leq \frac p2(2\dim \lie u)+\frac {(p-1)}2  \dim \lie m,
$$
and
\begin{align*}
&\dim W_0-\dim H-m_0(W_0)
\\&\qquad\geq p(2\dim \lie u+\ell)+(p-1)\dim\lie m-(2\dim \lie u+\ell)-\dim Z-m_0(W_0)\\
&\qquad=\delta:=(p-2)\dim \lie u+(p-1)\ell+\left(\frac{p-1}2\dim\lie m-\dim Z\right).
\end{align*}
The first two terms of $\delta$ are non-negative, as is the third since $\dim\lie m\geq 2\dim Z$.   If $p>2$, one  easily sees that $\delta\geq 2$ so that (1)  holds, and (3) and (4) hold  by  Lemma \ref{lem:delta}.  Now assume that $p=2$ so that $\delta=\ell+ (1/2)\dim \lie m-\dim Z$. If $H_s$ is not trivial, then (1),  (3) and (4) hold  since $\delta\geq 1$ and $\dim H/U\geq 2$. So we   are left with the case that $H^0=Z$ is a torus.
Since $W_0$ is an almost faithful $Z$-module and $(W_0,Z)$ is  orthogonal, hence stable, $W_0$ has FPIG.  Theorem  \ref{N.good.tori} then shows that $N_0$ is
CIFR and (2) holds. We may assume that $Z\subset T$ where now $T$ is the maximal torus of $G$.
  Then $(W_0,Z)$ is a slice representation of $(\lieg,T)$ and if $\Rank G>1$, then (3)  and (4) hold  by Proposition \ref{prop:torus.on.lieg} and Corollary \ref{cor:torus.codim.sings}.

If $\lieg=\lie sl_2$ and $p=2$, then $V=2\,\lieg$ has a slice representation $(W_0,\C^\times)$ where the weight vector is $(2,-2)$. Then $\codim_{W_0}\NN(W_0)=1$, and  $N_0$ has dimension three and has a singular point at the origin. Thus (3) and (4) fail.
\end{proof}

\begin{corollary}\label{cor:factorial}
Suppose that $p>2$ or that $\Rank G>1$. Let $(S,H)$ be a symplectic slice representation of $N$.
\begin{enumerate}
\item $N_0$ is factorial. In particular,
$N_V$ is factorial.
\item $\C[N_0]^*=\C^\times$.
\item If the character group $\chi(H)$ is trivial, then $N_0\git H$ is factorial. In particular,
$N_V\git G$ is factorial.
\end{enumerate}
\end{corollary}

\begin{proof}
If follows from Lemma \ref{lem:lie.alg.main}(3) and induction over symplectic slice representations  that
$(N_0)_\sing$ has codimension at least $4$ in $N_0$.
Since $N_0$ is a complete intersection, $N_0$ is locally factorial \cite[ Expos\'e XI Corollaire 3.14]{GrothendieckFactorial}. Since $N_0$ is a cone, it follows that $\C[N_0]$ is a UFD \cite[Ch.\ 2, Ex. 6.3(d)]{Hartshorne}, and we have (1). Let $f\in\C[N_0]^*$. Then $\C^\times$ acts by the scalar action on $N_0$, and by the argument in \cite[Proposition 1.3]{KnopKraftVust}, $f\circ t=\chi(t)f$ for a character $\chi$ of $\C^\times$. Hence $f$ is homogeneous. This forces $f$ to be a constant and we have (2).    Now let $f\in \C[N_0]^H$. Then $f$ factors uniquely in $\C[N_0]$ as a product $f_0\cdot f_1\cdots f_m$ where $f_0$ is a unit (hence a constant), and $f_1,\dots,f_m$ are irreducible and transform by elements of $\chi(H)$.
If $\chi(H)$ is trivial,  the $f_i$ are $H$-invariant and we have (3).
\end{proof}

\begin{corollary}\label{cor:no.codim.1.strata}
If $p>2$ or $\Rank G>1$, then $V\git G$ has no codimension one strata.
\end{corollary}

\begin{proof}
Let $(W,H)$ be the slice representation of a codimension one stratum where $\dim H>0$. Then $\dim W_0\git H=1$ so that $\codim_{W_0}\NN(W_0)=1$ which contradicts Lemma \ref{lem:lie.alg.main}(4). Hence $H$ is finite cyclic and $\dim W_0=1$. But $H$ is a subgroup of a maximal torus $T$ of
$G$ 
in which case $\dim W_0$ is even, a contradiction.
\end{proof}

In the following two results we consider the adjoint representation of a semisimple group $G$. Since the action of $G$  on $\lieg$ factors through the adjoint group,
there is no harm in assuming that $G$ is a product of simple groups

\begin{theorem}\label{thm:p.lieg}
Let $G$ be semisimple and $V=p\,\lieg$ for $p\geq 2$. Let
$N = N_V\subset V\oplus V^*$ be the shell.
\begin{enumerate}
\item $N$ is
CIFR and $V$ is $1$-large.
\item If $p>2$ or $G$ contains no simple factor of rank $1$, then
$N$ and $N\git G$ are factorial and $V$ is $2$-large.
\end{enumerate}
\end{theorem}

\begin{proof}
We may assume that  $G=G_1\times\dots\times G_m$ is
a decomposition of $G$ into simple factors. Then $\lieg=\lieg_1\oplus\cdots\oplus\lieg_m$ and correspondingly $N=N_1\times\cdots\times N_m$. By Lemma \ref{lem:lie.alg.main} and Theorem \ref{thm:use.E_m}, each $N_i$ is
CIFR, hence so is $N$. Since $V$ is orthogonal and $1$-modular, it is  $1$-large   and we have (1).

Now assume the hypotheses of (2).
By Corollary  \ref{cor:factorial}, each
 $N_i$ is factorial, hence so are
 $N$ and $N\git G$. Moreover, by   Proposition \ref{prop:props.of.N}
 (6) and Corollary \ref{cor:no.codim.1.strata},
 each $p\,\lieg_i$ is $2$-modular and  $2$-principal. Hence  $V$ is $2$-large.
\end{proof}

We also have the following, which we will need in Section \ref{subs:tan.cone}.

\begin{corollary}\label{cor:arbitrary.H}
Let $G=G_1\times\dots\times G_m$  be as above and $p\geq 2$. Let  $H=H_1\times\dots\times  H_m$ be a
reductive subgroup of $G$ where $H_i\subset G_i$, $i=1,\dots,m$.   Let
$W=p\,\lieh\oplus (p-1)(\lieg/\lieh)$ and define $W_0$ and  $N_0$ as usual. Then every symplectic slice representation of $N_0$ satisfies (1) or (2) of Lemma \ref{lem:lie.alg.main}, hence
$N_0$ is CIFR.
\end{corollary}

\begin{proof}
It is enough to consider the case that $G$ is simple so that we are in the situation of Lemma \ref{lem:lie.alg.main}. The proof of the lemma  shows that (1) or (2) holds even if $(W\oplus W^*,H)$ is not a symplectic slice representation of $N$.  The same is true for any symplectic slice representation of   $N_0$. By Theorem \ref{thm:use.E_m}, 
$N_0$ is CIFR. 
\end{proof}


\section{$N_V\git G$ has symplectic singularities.}\label{sec:symplectic.sings}
Let $N = N_V$.
We present  conditions that are sufficient for  $N\git G$ to have symplectic singularities,  different than our criteria in \cite{HerbigSchwarzSeaton2} where we required $V$ to be $3$-large or $2$-large with $(N\git G)_\pr=(N\git G)_\sm$.
Henceforth, $(*)$ will denote the following  condition.

 Let $(S,H)$ be a
 symplectic slice representation of $N$. Then
 UTCLS we have either
\begin{enumerate}
\item $m_0(W_0)<\dim W_0-\dim H$, or
\item $H^0$ is a torus and $W_0$ has FPIG.
\end{enumerate}

\begin{remark}\label{rem:star.p.lieg}
By Proposition \ref{prop:N.and.F}, $(*)$ implies that $V$ is $1$-modular and
$N$ is a normal complete intersection with FPIG.
\end{remark}

By  Theorems \ref{N.good.tori} and   \ref{thm:use.E_m}, $(*)$ implies that $N$ is
CIFR, hence $\dim N\git G=2\dim V-2\dim G$ and $N\git G$ has rational singularities \cite{Boutot}.

Let $X$ denote $N\git G$. The algebra $\C[X]$ is graded, and it is normal and Cohen-Macaulay since $X$ has rational singularities. Recall that $X$ is \emph{graded Gorenstein\/} if  the canonical module $\omega_X$ (which has a grading) is a free   $\C[X]$-module with generator of degree $\dim X$.

\begin{theorem}\label{thm:symplectic.sings}
If $(*)$ holds then $X$  is graded Gorenstein  and has symplectic singularities.
\end{theorem}

We give a proof of the theorem after some preliminary results.

\begin{proposition}\label{prop:graded.Gor}
Assume $(*)$. Then $X$ is graded Gorenstein.
\end{proposition}

\begin{proof}
Let  $U=X_\pr\subset X_\sm$. Then $\codim_X (X\setminus U)\geq 2$ (since all strata of $X$ are even dimensional).  Now $L=\bigwedge^{\dim X}(T^*U)$ is a line bundle, and in \cite[Proof of Theorem 4.6]{HerbigSchwarzSeaton2} we construct a nowhere vanishing section $\sigma$ of $L$. It follows that $\Gamma(U,L)\simeq \C[U]=\C[X]$, so that $X$ is Gorenstein. By construction, the degree of $\sigma$ equals the dimension of $X$, so that $X$ is graded Gorenstein.
\end{proof}

Assume that $N=N_V$ has   principal isotropy groups, e.g., $V$ is $1$-modular.  Let  $(S=W\oplus W^*,H)$ be a symplectic slice representation of $N$. We say that $(S,H)$ is \emph{proper\/} if $H\neq G$ and we say that $(S,H)$ and the stratum $N_{(H)}$ of $N$ are \emph{subprincipal\/} if every proper symplectic slice representation of  $(S,H)$ is principal.

\begin{lemma}\label{lem:1-large}
Let $(S,H)$ be as in $(*)$. If the corresponding stratum  is subprincipal, then $(W,H)$ is 1-large.
\end{lemma}

\begin{proof}
Let $W=W^H\oplus W_0$ be a Lagrangian $H$-submodule of $S$. We know that $W_0$  is $1$-modular and $N$ has FPIG.
Since $(S,H)$ is subprincipal,
 the complement  of the principal orbits  in $W_0$ is $\NN(W_0)$. By $(*)$, $\NN(W_0)\neq W_0$. It follows that $W_0$ has FPIG and is $1$-principal, hence $1$-large.
\end{proof}

\begin{corollary}\label{cor:smooth.points}
If $(*)$ holds, then $X_\sm=X_\pr$ and $X_\sm$ carries a symplectic form.
\end{corollary}

\begin{proof}
Let  $(S,H)=(W\oplus W^*,H)$ be  a subprincipal symplectic slice representation of $N$. Then $W$ is $1$-large and by \cite[Lemma 2.3]{HerbigSchwarzSeaton},  $N_W\git H$ is singular along $(N_W\git H)_{(H)}$. Thus $X$ is singular along the  union of the  subprincipal strata, hence along their closure, which is the complement of $X_\pr$. Thus $X_\sm=X_\pr$.  Finally, \cite[Cor.\ 3.18]{HerbigSchwarzSeaton2}  shows that   $X_\pr$ carries a (holomorphic) symplectic form.
\end{proof}

\begin{proof}[Proof of Theorem \ref{thm:symplectic.sings}]
We know that  $X$ is (graded) Gorenstein with rational singularities and has a symplectic form on its smooth locus.  By   \cite[Theorem 6]{NamikawaExtension}, $X$ has symplectic singularities.
\end{proof}


\section{Classical representations of the classical groups}
\label{sec:Examples}

We consider 1-modular classical representations $(V,G)$ of the  classical groups (and some not so classical) as in \cite[Theorem 3.5]{HerbigSchwarz}.
Let $N = N_V\subset V\oplus V^*$.
We show that
 for any symplectic slice representation
$(S = W\oplus W^*,H)$ of $N$
the shell  $N_W$ is
CIFR. We leave it as an exercise for the reader to show that each $N\git G$ has symplectic singularities. These results improve upon those of \cite{HerbigSchwarzSeaton2} where $V$ is required to be at least $2$-large for $N\git G$ to have symplectic singularities.

The following improves upon  Theorem 5.1 of \cite{CapeHerbigSeaton}.

\begin{theorem}\label{thm:kC^n,SO(n)}
Let $(V,G)=(k\C^n,\SO_n(\C))$, $n\geq 2$. Then $V$ is $1$-modular if and only if  $k\geq n-1$ in which case  the shell $N\subset V\oplus V^*$ is
CIFR.
 \end{theorem}

 \begin{proof}
For the statement about $1$-modularity see \cite[Theorem 11.18]{GWSlifting}.
So assume that $k\geq n-1$.  Since $V$ is stable, it is also $1$-large, so that  we already know that $N$ is a complete intersection with FPIG. By Proposition \ref{prop:orthog.reps} every nontrivial symplectic slice representation  of $N$ is,  up to trivial factors,   of the same form as $(V\oplus V^*,G)$ (with a smaller $n'<n$).
Hence we may assume by induction on $n$  that $N\setminus\NN(N)$ has rational singularities.
If $n$ is odd, then Theorem \ref{thm:rational.sings.orthogonal}(4)  applies.
 Hence it is enough to consider  the case $(V,G)=((2n-1)\C^{2n},\SO_{2n})$ when $\Ann E_m\subset\NN(V^*)$, with $E_m$ as defined in Section \ref{subsec:JetShellLinSubspace}, which implies that  $\dim E_m=\dim G=(1/2)\dim V$. Then $\dim \Ann E_m=(1/2)\dim V$ so that $\Ann E_m$ is a maximal isotropic subspace of $V^*$. Thus $E_m$ is maximally isotropic, hence  $E_m\subset\NN(V)$. Note that the projection of $E_m$ to any component copy of $\C^{2n}$ has dimension $n$.

Let $x_0=(y_1,\dots,y_{2n-1})$ where $y_i\in\C^{2n}$ for all $i$. Let us assume that some $y_i$ is not zero, say $y_1$. Let $W$ denote the first copy of $\C^{2n}$. Then the projection of $Gx_0$ to $W$    has dimension $2n-1>n$. This is bigger than the projection of $E_m$. Thus we must have $x_0=0$. Similarly, $x_1=\dots=x_m=0$. Hence $\Ann E_m\subset\NN(V^*)$ is not possible.
By Remark \ref{rem:Em.not.in.nullcone} $\NN(N)$ is irrelevant, hence   $N$ has rational singularities and is CIFR by Proposition \ref{prop:irrelevant}.
 \end{proof}

 Now we consider copies of the $(G=\SP_{2n})$-module $\C^{2n}$. Let $K$ be a maximal compact subgroup of $G$.   Let $T$ be a maximal torus of $G$ which is the complexification of a maximal torus $T_0$ of $K$. Then the weights of $V$ are $\pm \epsilon_i$, $i=1,\dots,n$, where the $\epsilon_i$ are the usual weights of $T\simeq ( \C^\times)^n$ acting  on $\C^n$.

 \begin{lemma}
 \label{lem.Sp.m0}
 Let $(V,G)=(k\C^{2n},\SP_{2n})$ . If $k\geq 2$, then $m_0(V)\leq kn$.
 \end{lemma}

 \begin{proof}
Let $L$ be a linear subspace of $\NN(V)$.
 If $k$ is even, then $V$ is orthogonal and this follows from Proposition \ref{prop:dim.L}. So assume that $k$ is odd and $\dim L>kn$. As in the proof of Proposition \ref{prop:dim.L}, we may assume that $L$ is $T$-stable  so that
  $$
  L=\bigoplus_{i=1}^n (L_{\epsilon_i}+L_{-\epsilon_i}).
  $$
  Then for some $\epsilon=\epsilon_i$, $\dim (L_\epsilon+L_{-\epsilon})>k$. Let $V'$ denote the sum of two copies of $\C^{2n}$ in $V$ and let $\pi\colon V\to V'$ be the $G$-equivariant projection. Let $L'=\pi(L)$. Then for some choice of $V'$, $\dim (L'_\epsilon++L'_{-\epsilon})>2$.  Note that $L'$ lies in $\NN(V')$ and that $V'$ is orthogonal.
 Now $V' \simeq  W\otimes_\R\C$ where $W$ is a real $K$-module and $V'_\epsilon+V'_{-\epsilon}$ is the complexification of a real two-dimensional $T_0$-stable subspace $W_0$ of $W$. Since $\dim(L'_\epsilon+L'_{-\epsilon})>2$, $L'$ intersects $W_0$ nontrivially. But any nonzero point of $W$ does not lie in the nullcone \cite{Birkes}. This is a contradiction. Hence $\dim L\leq kn$.
 \end{proof}

 \begin{theorem}
 Let $(V,G)=(k\C^{2n},\SP_{2n})$. Then $V$ is $1$-modular if and only if $k\geq 2n+1$ in which case the shell $N$ is
 CIFR.
 \end{theorem}

 \begin{proof}
 By \cite[Theorem 11.20]{GWSlifting}, $V$ has FPIG only if $k\geq 2n$ and is $1$-modular only if $k\geq 2n+1$ which we now assume.    The symplectic slice representations of $N$ are all, up to trivial factors, of the same form as $V\oplus V^*$, so
 we may again assume that $N\setminus\NN(N)$ has rational singularities. If $k$ is even, then $V$ is orthogonal, $k\geq 2n+2$ and Theorem \ref{thm:rational.sings.orthogonal}(4) applies. So assume that $k$ is odd. By Lemma \ref{lem.Sp.m0}, $m_0(V)\leq kn$. If $k>2n+1$, then Theorem \ref{thm:use.E_m}(4)  applies so we only need consider the case $k=2n+1$.
 Then
 $\Ann E_m\subset\NN(V^*)$ implies that $\dim E_m=\dim \Ann E_m=(1/2)\dim V=\dim G$. Hence $\Ann E_m$ is maximal isotropic in $V^*$ so that $E_m$ is maximal isotropic in $V$, and $E_m\subset\NN(V)$. Thus we may assume that $E_0=k\,\C^n\subset V$ is the subspace corresponding to the weights $\epsilon_1,\dots,\epsilon_n$.
 As in the proof of Theorem \ref{thm:kC^n,SO(n)}, this implies that we cannot have $\Ann E_m\subset\NN(V^*)$ so that $N$ is
 CIFR.
  \end{proof}

 Note that special cases of the two results above are $(2\C^3,\SO_3(\C))=(2\lie {sl_2},\SL_2(\C))$ and $(3\C^2,\SL_2(\C))$.

\begin{theorem}
Let $(V,G)=(k\C^7,G_2)$. Then $V$ is $1$-modular if and only if $k\geq 4$ in which case $N$ is
CIFR.
\end{theorem}

\begin{proof}
From \cite[Theorem 11.21]{GWSlifting} we see that $V$ is not $1$-modular if $k\leq 3$. We assume that $k\geq 4$ in which case $V$ is $2$-large.  The nontrivial parts of the proper nontrivial slice representations of $V$ are
$$
((k-1)(\C^3\oplus (\C^3)^*),\SL_3(\C))\text{ and }(2(k-2)\C^2,\SL_2(\C))
$$
 and Theorem \ref{thm:rational.sings.orthogonal}(4) applies for all the slice representations.
\end{proof}

\begin{theorem}
Let $(V,G)=(k\C^8,\Spin_7(\C))$. Then $V$ is $1$-modular if and only if $k\geq 5$ in which case $N$ is
CIFR.
\end{theorem}

\begin{proof}
From  \cite[Theorem 11.21]{GWSlifting} we see that $V$ is not $1$-modular if $k\leq 4$, so we assume that $k\geq 5$ in which case $V$ is $2$-large. The proper slice representations of $V$ are, up to trivial factors, the slice representations of  $((k-1)\C^7,G_2)$, so
we may assume that $N\setminus\NN(N)$ has rational singularities.

If $k\geq 6$, then  Theorem  \ref{thm:rational.sings.orthogonal}(4) applies, so the only interesting case is $k=5$.  Everything is OK if $\dim E_m\leq 20$ for then $\dim \Ann E_m\geq 20=(1/2)\dim V$ and we can argue as above (a minimal nonzero orbit in $\C^8$ has dimension  $7$). That leaves the case $\dim E_m=21=\dim G$ in which case $\dim\Ann E_m=19<(1/2)\dim V$. If $\Ann E_m\subset\NN(V^*)$, then the projection of $\Ann E_m$ to each copy of $(\C^8)^*$ has dimension 3 or 4 so that the projection of $E_m$ to each copy of $\C^8$ has dimension 4 or 5. As before, this leads to a contradiction.
\end{proof}

We give a new proof of the following which is also established in   \cite[Section 6.3]{HerbigSchwarzSeaton2}.

 \begin{theorem}
 Let $(V,G)=(R_1+R_2,\SL_2(\C))$. Then  $N$ is
 CIFR.
 \end{theorem}

 \begin{proof}
 That $V$ is $1$-large follows from \cite[Theorem 3.4]{HerbigSchwarz}.
Let  $T=\C^\times\subset G$ be the diagonal torus.
 The only nontrivial non-principal symplectic slice representation of  $N$ is a $T$-module with nonzero weights $(1,1,-1,-1)$, hence $N\setminus\NN(N)$ has rational singularities.
  The dimension of $\NN(V^*\simeq V)$ is 3, so that if $\dim E_m<3$, then $\dim \Ann E_m\geq 3$ and since $\NN(V^*)$ is not a vector space, $\Ann E_m$ contains points outside of $\NN(V^*)$. Hence we only need to worry about the case $\dim E_m=3$ and   $L=\Ann E_m\subset\NN(V^*)$.

  Let $(x,y)$ be a basis of $R_1$ with weights $(1,-1)$ relative to
  $T$.
   Let $(e,h,f)$ be the basis of $R_2$ with weights (2,0,-2).    Let $L_0$ denote $\C\cdot x\oplus\C\cdot e$. Then $\NN(V^*)$ is $G\cdot L_0$.  Let $L_2$ be the projection of $L$ to $R_2$. If $\dim L_2=0$, then $L=R_1$. Otherwise,  $\dim L_2=1$ since $\dim\NN(R_2)=2$ and $\NN(R_2)$ is not linear. It follows that for some $g\in G$, $L=g\cdot L_0$.
   If $L=\Ann E_m=R_1$, then $E_m$ is $R_2$ and every $x_i$ is in $R_2$ in which case  the corresponding points in $N_m$ have  dimension less than $(m+1)\dim N$. In the remaining case, we may assume that
   $$
   E_m=\Ann L_0=\C\cdot y\oplus\C\cdot h\oplus\C\cdot f.
   $$
If $x_0\neq 0$, then $E_0\subset E_m$ forces $x_0$  to be a multiple of $f$ and then $E_0=\C\cdot h\oplus \C\cdot f$ and $\lieg_0=\C\cdot f$. Now $\lieg_0$ maps $R_1+R_2$ into $E_0$. So $E_1=E_2=\cdots =E_m=E_0$ and we never get $\dim E_m=3$. Thus $x_0=0$, but then the same scenario repeats with $x_1$, etc.   Hence $N$ is
CIFR.
 \end{proof}

Let $(V,G)=(p\C^n\oplus q(\C^n)^*,\GL_n(\C))$ where we may assume by choosing a Lagrangian submodule of $V\oplus V^*$ that $p\geq q$. Then $V$   has non-finite principal isotropy groups if $q<n$ and is $1$-large if $q\geq n$. Using Lemma \ref{lem:mu.formula} and Remark \ref{rem: change.V.UTCLS} we see that $V$ is $1$-modular if and only if $p+q\geq 2n$.

\begin{lemma}
\label{lem.GL.m0}
Let $V$ be as above where $p\geq q\geq n$. Let $L$ be a linear subspace of $\NN(V)$. Then $\dim L\leq np$.  If $p>q$ and $\dim L=pn$, then $L=p\,\C^n$.
\end{lemma}

\begin{proof}
Since $W=q(\C^n\oplus(\C^n)^*)$ is orthogonal, the projection of $L$ to $W$ has dimension at most $qn$ by Proposition \ref{prop:dim.L}.
Let $W'\simeq (p-q)\C^n$ be the complement of $W$. The projection of $L$ to $W'$  has dimension at most $(p-q)n$, so that $\dim L\leq pn$. If there is equality and $p>q$, then the projection of $L$ to $W'$ being surjective implies that the projection of $L$ to $q(\C^n)^*$ is $0$ so that $L=p\,\C^n$.
\end{proof}

\begin{theorem}
Let $(V,G)=(p\C^n\oplus q(\C^n)^*,\GL_n(\C))$. Then $V$ is $1$-modular if and only if $p+q\geq 2n$ in which case $N$ is
CIFR.
\end{theorem}

\begin{proof}
By choosing a Lagrangian submodule
 we may assume that $p\geq q\geq n$.
Any symplectic slice representation of $N$ is either trivial or of the same form as $V\oplus V^*$, up to trivial factors and having a smaller $n$.
Thus we may assume that $N\setminus\NN(N)$ has rational singularities. Let $L=\Ann E_m$ be a linear subspace of $\NN(V^*)$.
By Lemma \ref{lem.GL.m0},
$$
\dim V^*-\dim G=n(p+q)-n^2=n(p+q-n)\geq  np\geq \dim L.
$$
Thus $\dim L\leq\dim V-\dim G$ with   equality only if $q=n$ and $\dim L=np$.

Suppose that $\Ann E_m$ has dimension $np$ and $q=n$. If $p>q$, then $\Ann E_m=p(\C^n)^*\subset V^*$ so that $E_m=q(\C^n)^*\subset V$ and  $\vec x
 \in V^{m+1}$ lies in $(q(\C^n)^*)^{m+1}$
 which has codimension greater than $(m+1)\dim G$ in $V^{m+1}$. By Corollary \ref{cor:rat.sings.by.codim}, $N$ has rational singularities. So we are left with the case $p=q=n$ and $\dim \Ann E_m=n^2$. Then $\Ann E_m$ is a maximal linear subspace on which the invariants of $V\simeq V^*$ vanish, hence so is $E_m$.  The case that $E_m=n\C^n$ or $E_m=n(\C^n)^*$ is easy since then the codimension of $X_{\vec r}$ is   greater than
 $(m+1)\dim G-\sum_{i=0}^mr_i$ unless
 $r_i=0$ for all $i$ in which case $E_m=0$, a contradiction. Thus $E_m$ projects to a proper nonzero subspace of each copy of $\C^n$ and $(\C^n)^*$. But any nonzero orbit of $G$ on $\C^n$ or its dual is the complement of the origin. This contradicts the form of $E_m$. Hence $N$ is
 CIFR.
 \end{proof}

 Finally, we consider the case $(V,G)=(p\,\C^n\oplus q(\C^n)^*,\SL_n(\C))$. Then $V$ is 1-modular if and only if $p+q\geq 2n-1$ \cite[Proposition 11.14]{GWSlifting}, which we now assume.

 \begin{lemma}
 Let $(V,G)$ be as above where we also assume that $p\geq q$ and $p-q\leq 1$.
Let $L$ be a maximal linear subspace of $\NN(V)$. Then $L=p\, W\oplus q\Ann W$ where $W$ is a linear subspace of $\C^n$ with $0<\dim W<n$.
\end{lemma}

 \begin{proof}
 Let  $L_0$ be the  projection of $L$ to $p\,\C^n$ and $L_1$ its projection to $q\,(\C^n)^*$.  Since the determinant function vanishes on $L_0$,  $L_0\subset p\,W$ where $W$ is the span of a basis of $L_0$ and has dimension at most $n-1$. It follows that a basis of  $L_1$ spans a subspace of $\Ann W$. If $W\neq 0$, then $L$ equals $p\,W\oplus q\Ann W$.   If $W=0$, then $L$ is not maximal:    as before,  $L_1\subset q\,\Ann W'$  where $\dim W'>0$ and then $L\subset p\,W'\oplus q\Ann W'$. Thus the maximal $L$ are as claimed.
 \end{proof}

 \begin{corollary}\label{cor:dim.E_m}
 Suppose that $E$ is a linear subspace of $V$ and $\Ann E\subset\NN(V^*)$. Then
 $$
 \dim E\geq qn+(p-q).
 $$
 \end{corollary}

 \begin{theorem}
 Let $(V,G)$ be as above. Then $V$ is $1$-modular if and only if $p+q\geq 2n-1$ in which case $N$ is
 CIFR.
 \end{theorem}

 \begin{proof}
It is easy to see that all proper nontrivial symplectic slice representations $(S,H)$ of $N$ are, up to trivial factors, of the same form with $H=\SL_{n'}$ for $n'<n$.
By induction we may assume that $N\setminus\NN(N)$ has rational singularities.
By choosing a Lagrangian submodule, we may assume that $p\geq q$ and $p-q\leq 1$. We suppose that $L=\Ann E_m\subset \NN(V^*)$ and by Corollary \ref{cor:dim.E_m}, $\dim E_m\geq qn+p-q$. If $q\geq n$ we get that $\dim E_m\geq n^2>\dim G$, which is impossible. So the only interesting case  to consider is $p=n$, $q=n-1$. Then a maximal $L\subset \NN(V^*)$ containing $\Ann E_m$ is of the form $n \Ann W\oplus (n-1)W$ where $n-k=\dim W$ and $0<k<n$, so that $E_m\supset n\, W\oplus (n-1)\Ann W$ has dimension at least $n^2-k$. Now suppose that $x_0$ projects nontrivially to $n\,\C^n$. Then $E_0$ projects surjectively onto some copy of $\C^n$. Since   $E_m$ already contains a copy of $n\,\C^{n-k}$, $E_m$ has dimension at least $n^2-k+k=n^2>\dim G$. Thus $x_0\in V_0=(n-1)(\C^n)^*$. Hence the non-rational locus of $N$ lies in $(V_0\times V^*)\cap N$.   Reversing the role of $V$ and $V^*$, we see that the non-rational locus of $N$ lies in $N'$ where $N'$ is the null cone of the shell of $V_0$. Since $\codim_N N'>\dim G$, Theorem \ref{thm:Budur3},  shows that $N$ has rational singularities, hence $N$ is
CIFR.
 \end{proof}


\section{Applications to representation and character varieties}
\label{sec:Applications}

 Let $\Sigma$ be a Riemann surface of genus $p\geq 2$ and let $G$ be a reductive complex algebraic group. Let $\pi$ denote $\pi_1(\Sigma)$. Then $\pi$ is the quotient of the free group on generators $a_1,b_1,a_2,\dots,a_p,b_p$ by the normal subgroup generated by
 $$
 [a_1,b_1][a_2,b_2]\cdots[a_p,b_p]
 $$
 where $[a_i,b_i]$ is the commutator $a_ib_ia_i\inv b_i\inv$. Let $\Hom(\pi,G)$  denote the set of homomorphisms from $\pi$ to $G$. This has a scheme structure as $\Phi\inv(e)$ where
 $$
 \Phi\colon G^{2p}\to G, \quad (g_1,h_1,g_2,\dots,g_p,h_p)\mapsto [g_1,h_1]\cdots [g_p,h_p].
 $$
 Now $\Phi$ is $G$-equivariant where $g\in G$ acts on $G$ by conjugation and on $G^{2p}$ by conjugation on each component. Hence $G$ acts on $\Hom(\pi,G)$ and we denote the quotient by $\X(\pi,G)$. The quotient $\X(\pi,G)$ is called a \emph{character variety\/} and $\Hom(\pi,G)$ is called a \emph{representation variety}, although they may not be varieties (but their irreducible components are). Let $\rho_0\in\Hom(\pi,G)$ denote the trivial homomorphism.

 We will show the following.

 \begin{theorem}\label{thm:char.vars}
 Let $\pi$ be as above and assume that $G$ is semisimple.
 \begin{enumerate}
\item $\Hom(\pi,G)$ is
CIFR and each irreducible component has dimension  $(2p-1)\dim G$.
\item $\X(\pi,G)$ has symplectic singularities and has dimension $2(p-1)\dim G$.
\end{enumerate}
Now suppose that $p>2$ or that every simple component of $G$ has rank at least $2$.
\begin{enumerate}
\addtocounter{enumi}{2}
\item The singularities of $\Hom(\pi,G)$ are in codimension
at least
 four and $\Hom(\pi,G)$ is locally factorial.
\item The singularities of $\X(\pi,G)$ are in codimension
at least
four and the irreducible component containing $G\rho_0$ is locally factorial. In particular, if $G$ is simply connected, then $\X(\pi,G)$ is locally factorial.
\end{enumerate}
 \end{theorem}
  \begin{remark}
 By \cite[Theorem 0.1]{Li}, the number of irreducible components of $\Hom(\pi,G)$ (and $\X(\pi,G)$) is the cardinality of the fundamental group $\pi_1(G)$
when $G$ is connected and semisimple. More generally, when $G$ is connected and reductive, the number of irreducible components of $\Hom(\pi,G)$
is the cardinality of the fundamental group of
a maximal connected  semisimple subgroup of $G$ by \cite[Proposition A.1]{LawtonRamras}.
 \end{remark}

 From \cite[Corollary 1.3]{FuNilOrb} we obtain the following. Note that locally factorial implies
 locally
  $\Q$-factorial.
 \begin{corollary}
Let 
$C$ 
be an irreducible component of $\X(\pi,G)$ which is locally factorial
with singularities in codimension at least 4.
Let $U$ be an open subset of 
$C$ 
containing a singular point. Then $U$ does not have a symplectic resolution. In particular, 
$C$ 
does not have a symplectic resolution.
 \end{corollary}

Theorem \ref{thm:char.vars}(2) was established for reductive groups of type A
 in \cite{bellamy2019symplectic} while Theorem \ref{thm:char.vars}(1)  was established for $\GL_n$ and $\SL_n$ in \cite{BudurRational}.
 The results for $\SL_n(\C)$ (and also $\GL_n(\C)$) rely  on the theory of quivers. In fact, one proves the results for $\GL_n(\C)$ and then deduces they hold for $\SL_n(\C)$.
 For arbitrary reductive $G$ we  go in the opposite direction and  prove our results first in the semisimple case and then deduce our results in the connected reductive case. We do not use the theory of quivers.
 As remarked  before, Theorem  \ref{thm:char.vars}(1), which is essential in our proof of Theorem \ref{thm:char.vars}(2),
  follows from  \cite{AizenbudAvniRepGrowth}, but requires  values of $p$ greater than $2$.

We now return to the case of  an arbitrary semisimple $G$.
By Remark \ref{rem:closed.orbits}, to show that $Y=\Hom(\pi,G)$ has rational singularities  it is enough to show that  
$Y$ has rational singularities
at any $y\in Y$ such that  $Gy$ is closed.
 We show that the tangent cone at such
 $y$ is
the product of a vector space and
the zero set of a moment mapping of the kind considered
in Section \ref{sec:Adjoint}.
In particular, the tangent cone has rational singularities. It follows that $\Hom(\pi,G)$ has rational singularities
at $y$, hence has rational singularities.
It is not hard to show (using the Campbell-Hausdorff formula)  that the  tangent cone at $\rho_0\in\Hom(\pi,G)$ is the zero set of the moment mapping corresponding to the $G$-module $p\,\lieg$. We know that this has rational singularities. By \cite[Theorem IV]{AizenbudAvniRepGrowth}, this is enough to show that $\Hom(\pi,G)$ has rational singularities.
We give here a different and short proof using tangent cones as above. To show that $\X(\pi,G)$ has symplectic singularities we modify the arguments  of \S \ref{sec:symplectic.sings} to apply to the case of character varieties.

\subsection{The tangent cone} \label{subs:tan.cone} 
We establish Theorem \ref{thm:char.vars}(1).
Let $Y$
denote the affine scheme $\Hom(\pi,G)$. Let $y\in Y$ where $Gy$ is closed. We determine the tangent cone $\TC_y(Y)$ of $Y$ at $y$.
 We follow the approach of \cite{Goldman84,Goldman85}. Let   $A$ denote the Zariski closure of the subgroup of $G$ generated by the components of $y$
 (as a point of $G^{2p}$).
 By \cite[Theorem 3.6]{Richardson88},   $G y$ is closed if and only if $A$ is reductive. Let $H$ denote the stabilizer of $y$. Then $H=Z_G(A)$. The Zariski tangent space $T_y$ to $Y$ at $y$ is $\Ker d\Phi_y$. By the proof of \cite[Proposition 3.7]{Goldman84} or \cite[Lemma 4.8]{AizenbudAvniRepGrowth}, the image of $d\Phi_y$ in $\lie g$ is $\lie h^\perp$ where the perpendicular is relative to the Killing form
 $\B$   of $\lieg$.
  Thus $T_y$ is isomorphic to $(2p-1)\lieg\oplus\lie h$ as 
  an $H$-module. Since $G  y$ is a closed orbit in $Y$, there is a Luna slice $S$ at $y$ whose Zariski  tangent space $S_y$ at $y$ is an $H$-stable complement to $B_y:= T_y(Gy)\simeq\lieg/\lie h$ in $T_y$. Thus $S_y\simeq 2p\,\lie h\oplus (2p-2)\lieg/\lieh$ as 
  an $H$-module.

Let $\rho\colon \pi\to G$ be the homomorphism corresponding to $y$. Then $\pi$ acts on $\lieg$ via $\rho$ and the adjoint action of $G$. We denote the corresponding $\pi$-module by $\lieg_\rho$. In terms of group cohomology \cite[1.3]{Goldman84}, $T_y= Z^1(\pi,\lieg_\rho)$   and  $B_y=B^1(\pi,\lieg_\rho)$ so that $H^1(\pi,\lieg_\rho)\simeq S_y$ as 
an $H$-module. Given  $u$, $v\in H^1(\pi,\lieg_\rho)$, their cup product is in $H^2(\pi,\lieg_\rho\otimes\lieg_\rho)$ which maps via
 $\B$
 to $H^2(\pi,\C)\simeq  \C$. Goldman shows that the resulting alternating form  $\omega$ is non-degenerate, i.e, it is a symplectic form on $H^1(\pi,\lieg_\rho)$. We can also map $\lieg_\rho\otimes\lieg_\rho$ to $\lieg_\rho$ via Lie bracket.   We get a symmetric
 bilinear form
$$
u,\ v\in  H^1(\pi,\lieg_\rho)  \mapsto  [u, v]\in H^2(\pi,\lieg_\rho)\simeq  H^0(\pi,\lieg_\rho^*)^*\simeq \lie h.
$$
The corresponding quadratic form is
\begin{equation}\label{eq:moment.mapping}\tag{$\#$}
u \in  H^1(\pi,\lieg_\rho)  \mapsto (1/2)[u, u]\in  \lie h\simeq\lie h^*.
\end{equation}
A necessary  condition for $u$ to be in $\TC(Y)_y$ is that $[u,u]$ vanishes \cite[\S 4]{Goldman85}.
Now
$S_y$ has an $H$-invariant orthogonal structure since $\lieg$ is an orthogonal $H$-module. By \cite[Lemma 3.10]{HerbigSchwarzSeaton2},  there is a Lagrangian  $H$-submodule $W$ of $S_y$. Then $W\simeq p\,\lieh\oplus (p-1)\lieg/\lieh$. We have the usual moment mapping $\mu\colon W\oplus W^*\to\lie h^*$, and  by
Corollary \ref{cor:arbitrary.H}, 
the shell $\mu\inv(0)$ is
CIFR and has dimension $(2p-1)\dim\lieg-\dim \lieg/\lie h$. In the Appendix we show that
$(\#)$
is a moment mapping, hence agrees with $\mu$. It follows that   $\TC_y(Y)$  is a subscheme of $N_y:= \mu\inv(0)\times\lieg/\lie h$ where the latter has dimension $(2p-1)\dim G$.  Since $N_y$ is reduced and irreducible and $\dim N_y \leq \dim_yY$, it follows that  $\TC(Y)_y=N_y$.
Since $N_y$ has rational singularities,    $\TC(Y)_y$ does also and $Y$ has rational singularities at $y$.
By Remark \ref{rem:closed.orbits},  $Y$ has rational singularities. 
 Moreover, $\dim Y=(2p-1)\dim G$ so that $Y$ is a complete intersection.
To establish Theorem \ref{thm:char.vars}(1) we only need to show that $Y$ has FPIG. 

Let $Y'=\{y\in Y\mid $ the closed orbit in $\overline{Gy}$ has positive dimensional isotropy$\}$.   Then $Y'$ is $G$-stable and closed. Let $y\in Y'$ such that $Gy$ is closed. Then $\TC_y(Y')\subset \TC(Y)_y=N_y$ and, since $\TC_y(Y')$ is the limit of tangents of curves in $Y'$ starting at $y$, $\TC_y(Y')\subset (N_y)'$ where $(N_y)'$ is defined relative to  the action of  $H=G_y$ on $N_y$. Since $N_y$ has FPIG, $(N_y)'$ has positive codimension in $N_y$, hence $Y'$ has positive codimension in $Y$ and $Y$ has FPIG. Now the principal isotropy group  of each $N_y$ is 
the center of $G$, which is therefore 
the principal isotropy group of each irreducible component of $Y$ 
and $Y$ has FPIG.

\subsection{The character variety} \label{subs:char.var}
We establish  Theorem \ref{thm:char.vars}(2). We may assume that $G$ is simple.
Let $Y=\Hom(\pi,G)$ and  $Z=Y\git G=\X(\pi,G)$. By Theorem \ref{thm:char.vars}(1) and Boutot's theorem \cite{Boutot}, $Z$ has rational singularities. Since $Y$ has FPIG, each irreducible component of $Z$ has dimension $2(p-1)\dim G$. It remains to show that $Z$ has symplectic singularities.

We argue along the lines of  \S \ref{sec:symplectic.sings}. Our first step is to show that $Z_\sm=Z_\pr$.  The quotient of $Y_\pr$ is $Z_\pr\subset Z_\sm$.
Let $Gy$ be a closed orbit in $Y$ such that the isotropy group $G_y$ is not principal, i.e., not 
the center of $G$. We need to show that the corresponding point $z\in Z$ is not smooth. Let $\OO_Y$ denote the structure sheaf of $Y$. Then $\TC_y(Y)$ is the variety of the associated graded ring of the local ring $\OO_{Y,y}$. Let $\mu$ and the Luna slice  $S$ at $y$ be as 
in Section \ref{subs:tan.cone}.
Then $\TC_y(S)=\mu\inv(0)$. By Luna's slice theorem, the mapping $S\git G_y\to Z$ is \'etale at $y$. Now taking $G_y$-invariants commutes with taking the associated graded ring, hence   the induced mapping  $\TC_y(S)\git G_y\to \TC_z(Z)$ is also \'etale at $y$.  By Corollary \ref{cor:arbitrary.H}  and Corollary \ref{cor:smooth.points}, $\TC_y(S)\git G_y$ is not smooth at $y$ so that $\TC_z(Z)$ is not smooth at $z$, hence neither is $Z$. Thus $Z_\sm=Z_\pr$.

By \cite[1.7--1.9]{Goldman84},  there is a holomorphic symplectic form on $Z_\pr$ which agrees with $\omega$ on $H^1(\pi,\lieg_\rho)$ where $\rho$ is a principal point. Thus there is a nowhere vanishing holomorphic volume form on $Z_\sm$, which implies
as in the proof of Proposition \ref{prop:graded.Gor} 
that $Z$ is Gorenstein. Then by \cite[Theorem 6]{NamikawaExtension}, $Z$ has symplectic singularities and we have established Theorem \ref{thm:char.vars}(2).

\subsection{Codimension of the singular stratum}
 Suppose that $p>2$ or that $G$ contains no simple factor of rank $1$. We prove parts (3) and (4) of Theorem  \ref{thm:char.vars}. As in Theorem \ref{thm:p.lieg},
 the singularities of
 every tangent cone
 $\TC_y(Y)$ at a closed orbit in $Y=\Hom(\pi,G)$
 are in codimension at least $4$.
 Hence $\TC_y(Y)$ is factorial by \cite[ Expos\'e XI Corollaire 3.14]{GrothendieckFactorial}
 and   $Y$ is factorial at $y$. Since the set of factorial points is open  \cite{boissiere2019} and the set of closed orbits is dense, $Y$  is locally factorial and we have (3). It follows from (3) that the singularities of $Z=\X(\pi,G)$ are in codimension
 at least
 four. Let
 $Y_0$ denote the irreducible component of $Y$ containing  $\rho_0$. Let $L$ be a $G$-line bundle on $Y_0$. Since $G$ is semisimple, $G$ acts trivially on the fiber $L_{\rho_0}$, hence $G$ acts trivially on $L$ in a neighborhood of $\rho_0$, hence on all of $Y_0$. Thus $\Pic_G(Y_0)=\Pic(Y_0)$, and it follows as in   \cite[Theorem 8.3]{Drezet} that $Y_0\git G$ is locally factorial, giving (4).

\subsection{The reductive case}
Suppose that $G$ is connected and reductive. Then $G=G_sC$ where $G_s$ is semisimple and $C$ is the connected center of $G$. Let $F=G_s\cap C$.
The action of
the subgroup
$F^{2p}$ on
$G_s^{2p}$ induces an action on $\Hom(\pi,G_s)$.
We have an isomorphism
\begin{equation}\label{eq:torus}\tag{$\dagger$}
\Hom(\pi,G)= \Hom(\pi,(G_s\times C)/F)\simeq(\Hom(\pi,G_s)\times C^{2p})/F^{2p}
\end{equation}
where $F^{2p}$ acts on $\Hom(\pi,G_s)$ as above and on $C^{2p}$ as a subgroup.

\begin{theorem}\label{thm:char.vars.red.gp}
Let $G$ be connected and reductive with  $G=G_sC$ and $F=G_s\cap C$ as above. Then the conclusions of Theorem \ref{thm:char.vars} hold
with the following changes.
 $$
 \dim\Hom(\pi,G)=(2p-1)\dim G_s+2p\dim C=(2p-1)\dim G+\dim C,
 $$
 $$
 \dim\X(\pi,G)=(2p-2)\dim G_s+2p\dim C=(2p-2)\dim G+2\dim C, \text{ and }
 $$
 in (4),  $\X(\pi,G)$ is locally factorial if $G_s$ is simply connected.
 \end{theorem}
\begin{proof}
In the following, (1), (2), etc.\ refer to the parts of Theorem \ref{thm:char.vars}. Now $C^{2p}$ is a symplectic variety via a symplectic form which is $F^{2p}$-invariant. The symplectic form on the smooth locus of $\X(\pi,G)$ is also  $F^{2p}$-invariant since for $y$, $y'\in\Hom(\pi,G_s)$ on the same $F^{2p}$-orbit, the corresponding $\pi$-modules $\lieg_\rho$ and $\lieg_{\rho'}$ are the same. By \eqref{eq:torus} and
 \cite[Proposition 2.4]{Beauville}, $\X(\pi,G)$ has symplectic singularities. With the changes in dimension,  (1) and (2)  follow. Part (3) is   clear since $\Hom(\pi,G)$ is a complete intersection with
 singularities in codimension at least $4$,
 and then (4) follows since we are dividing by $G_s$ which has trivial character group.
\end{proof}

It was pointed out to us by Sean Lawton that an immediate consequence of Theorem \ref{thm:char.vars.red.gp} is the generalization of
\cite[Corollary 8(4)]{LawtonSikora} to the case of $G$ connected and reductive. The argument is identical to that in \cite{LawtonSikora} using
the fact that $\X(\pi,G)$ is normal. This in particular implies that the regular functions on the character variety are rational functions
in characters.

\section{Appendix}
\label{sec:Appendix}
We summarize some results in \cite[1.7--1.8]{Goldman84} and \cite{AtiyahBott} which we use  to show that the mapping \eqref{eq:moment.mapping} of \S \ref{sec:Applications} is a moment mapping.

Let $\Sigma$ be our compact Riemann surface of genus $p\geq 2$ with fundamental group $\pi$ and universal cover $\widetilde \Sigma$. Let $G$ be semisimple, and let $x\in\Hom(\pi,G)$ where $Gx$ is closed. Let $\rho\colon\pi\to G$ be the corresponding homomorphism and
let
$\lieg_\rho$ denote $\lieg$ with the action $\Ad\circ\rho$.  Since $\widetilde\Sigma\to\Sigma$ is a principal $\pi$-bundle we have various associated bundles. Let
$V_\rho$ denote the   flat vector bundle $\widetilde\Sigma\times^\pi\lieg_\rho$ with flat connection $\nabla_\rho$.   We have a flat principal $G$-bundle $P=\widetilde \Sigma\times^\pi G$ with flat connection $\nabla_P$ where the action of $w\in\pi$ on $G$ is via left-multiplication by $\rho(w)$. We have a flat bundle of groups $\widetilde G=\tilde\Sigma\times^{\operatorname{conj}_\pi}G$ where $w\in \pi$ acts on $G$ via conjugation with $\rho(w)$.  The global sections  $\G$ of $\widetilde G$, the gauge group of $P$, act  on
the left of
$P$ as principal bundle automorphisms.

Let $A^*(\Sigma,V_\rho)    =\E^*(\Sigma)\otimes \Gamma(V_\rho)$ where $\E^*(\Sigma)$ is the algebra of differential forms on $\Sigma$ and $\Gamma(V_\rho)$ is the space of
smooth
sections of $V_\rho$.   Then from $\nabla_\rho$ and the usual exterior derivative we get a differential $d_\nabla\colon
A^i(\Sigma,V_\rho)\to A^{i+1}(\Sigma,V_\rho)$. Let $Z^i(\Sigma,V_\rho)$ and $B^i(\Sigma,V_\rho)$ denote the kernel and image, respectively, of $d_\nabla$ in $A^i(\Sigma,V_\rho)$, and let $H^i(\Sigma,V_\rho)$ denote the quotient.  Then   $H^1(\pi,\lieg_\rho)\simeq H^1(\Sigma,V_\rho)$ and similarly for $H^2$.

Let $\sigma$, $\tau\in A^1(\Sigma,V_\rho)$. Then $\sigma\wedge\tau\in A^2(\Sigma,V_\rho\otimes V_\rho)$, and using the killing form
$\B$,
we obtain a two form $\B_*(\sigma\wedge\theta)\in \E^2(\Sigma)$. Then $\int_\Sigma \B_*(\sigma\wedge\theta)$ gives a symplectic form   $\omega^{(\B)}$ on
$A^1(\Sigma,V_\rho)$. Restricted to $Z^1(\Sigma,V_\rho)$, the form  $\omega^{(\B)}$ vanishes when either argument is in $B^1(\Sigma,V_\rho)$, and thus we obtain a $2$-form   (also denoted $\omega^{(\B)}$) on $H^1(\Sigma,V_\rho)\simeq H^1(\pi,\lieg_\rho)$ which agrees with the symplectic form $\omega$ constructed in \ref{subs:tan.cone} above.

Note that the Lie algebra of $\G$ is $\tilde\lieg=A^0(\Sigma,V_\rho)$. Let $\lie A$ denote the space of smooth connections on $P$ and let $F(A)\in A^2(\Sigma,V_\rho)\simeq A^0(\Sigma,V_\rho)^*=\tilde\lieg^*$ denote the curvature of $A\in \lie A$. The kernel of $F$ is the space of flat connections $\lie F\subset\lie A$.  Now $\lie A$ is an affine space (infinite dimensional) since $\lie A=A+A^1(\Sigma,V_\rho)$ for any $A\in\lie A$.
Moreover,  $\omega^{(\B)}$ is   invariant under translation by $A^1(\Sigma,V_\rho)$.
Thus  $\omega^{(\B)}$ gives a closed non-degenerate 2-form on $T(\lie A)$.  Now $\G$ acts on $\lie A$ and by \cite[p.\ 587]{AtiyahBott}, $F\colon\lie A\to \tilde\lieg^*$ is a moment mapping relative to $\omega^{(\B)}$.

Now we consider what happens near $\nabla_P$. Since $H=G_x$ and $\rho(G)$ commute, $H$ is a subgroup of $\G$.
The tangent space to the orbit $\G\cdot \nabla_P$ is $B^1(\Sigma,V_\rho)$. Since $\nabla_P$ is flat,   $F$ vanishes on $\nabla_P+B^1(\Sigma,V_\rho)$. Then for $u\in Z^1(\Sigma,V_\rho)$,
we have by \cite[Lemma 4.5]{AtiyahBott} that
$F(\nabla_P+u)=(1/2)[u,u]$ where the bracket only depends upon the image of $u$ in $
H^1(\Sigma,V_\rho)\simeq H^1(\pi,\lieg_\rho)$. This bracket is the same as taking the bracket we defined in \eqref{eq:moment.mapping}.
Thus \eqref{eq:moment.mapping}  is a moment mapping.


\bibliographystyle{amsalpha}
\bibliography{HSS-rational}

\end{document}